\documentclass[12pt]{article}

\usepackage{geometry}
\usepackage{fullpage}
\usepackage{microtype}
\usepackage[utf8]{inputenc}
\usepackage{setspace}

\usepackage[round]{natbib}
\usepackage{hyperref}
\usepackage[title]{appendix}

\usepackage{graphicx}
\usepackage{floatrow}

\usepackage{amsmath, amssymb, amsthm}
\theoremstyle{plain}
\newtheorem{thm}{Theorem}[section]

\newtheorem{lem}[thm]{Lemma}
\newtheorem{prop}[thm]{Proposition}
\newtheorem{corol}[thm]{Corollary}
\theoremstyle{definition}

\newtheorem{defn}{Definition}
\newtheorem{exm}{Example}
\newtheorem{innercustomexm}{Example}
\newenvironment{customexm}[1]
  {\renewcommand\theinnercustomexm{#1}\innercustomexm}
  {\endinnercustomexm}

\newtheorem{assmp}{Assumption}
\theoremstyle{remark}

\newcommand{\E}{\mathbb{E}}
\newcommand{\I}{\mathbb{I}}
\newcommand{\var}{\operatorname{var}}
\newcommand{\cov}{\operatorname{cov}}

\newcommand{\pr}{\mathsf{P}}
\newcommand{\dd}{\mathrm{d}}

\newcommand{\tr}{\operatorname{tr}}
\newcommand{\diag}{\operatorname{diag}}
\newcommand{\tsp}{\mathsf{T}}
\newcommand{\rN}{\mathcal{N}}
\newcommand{\emax}{\operatorname{\bar{e}}}
\newcommand{\emin}{\operatorname{\underline{e}}}
\newcommand{\R}[1]{\mathbb{R}^{#1}}

\DeclareMathOperator*{\argmax}{arg\,max}

\begin{document}

\title{Confidence Regions Near Singular Information and Boundary Points With
Applications to Mixed Models}

\author{Karl Oskar Ekvall \quad Matteo Bottai\\
{\normalsize Division of Biostatistics, Institute of Environmental Medicine, Karolinska Institutet} \\
{\tt \normalsize karl.oskar.ekvall@ki.se \quad matteo.bottai@ki.se}}
\date{\normalsize \today}

\maketitle

\onehalfspacing

\abstract{
  \noindent We propose confidence regions with asymptotically correct uniform
    coverage probability of parameters whose Fisher information matrix can be
    singular at important points of the parameter set. Our work is motivated by
    the need for reliable inference on scale parameters close or equal to zero in
    mixed models, which is obtained as a special case. The confidence regions are
    constructed by inverting a continuous extension of the score test statistic
    standardized by expected information, which we show exists at points of
    singular information under regularity conditions. Similar results have
    previously only been obtained for scalar parameters, under conditions stronger
    than ours, and applications to mixed models have not been considered. In
    simulations our confidence regions have near-nominal coverage with as few as
    $n = 20$ observations, regardless of how close to the boundary the
    true parameter is. It is a corollary of our main results that the proposed
    test statistic has an asymptotic chi-square distribution with degrees of
    freedom equal to the number of tested parameters, even if they are
    on the boundary of the parameter set.
}

\section{Introduction}
  In mixed models, the importance of a random effect is often assessed by
  inference on a variance or scale parameter. A parameter near zero typically
  indicates a weak effect and many tests for whether a variance is equal to zero
  have been proposed \citep{Stram.Lee1994, Stram.Lee1995, Lin1997,
  Stern.Welsh2000, Hall.Praestgaard2001, Verbeke.Molenberghs2003,
  Crainiceanu.Ruppert2004, Zhu.Zhang2006, Fitzmaurice.etal2007, Greven.etal2008,
  Giampaoli.Singer2009, Saville.Herring2009,Sinha2009, Wiencierz.etal2011,
  Drikvandi.etal2013, Qu.etal2013, Wood2013, Baey.etal2019, Chen.etal2019}. In
  addition to being of practical interest, this case is of theoretical interest
  because the parameter is a boundary point of the parameter set and,
  consequently, asymptotic distributions of common test statistics are
  non-standard. For example, the asymptotic distribution of the likelihood ratio
  test statistic for a variance equal to zero is a non-trivial mixture of
  chi-square distributions \citep{Self.Liang1987, Geyer1994, Stram.Lee1994},
  whereas for a strictly positive variance it is a chi-square distribution with
  one degree of freedom. More generally, the asymptotic distributions of common
  test statistics under a sequence of parameters tending to a boundary point as
  the sample size increases, can be different depending on the rate of that
  convergence \citep{Rotnitzky.etal2000, Bottai2003}. While this need not be an
  issue when testing a point null hypothesis, it complicates more ambitious
  inference: coverage probabilities of confidence regions obtained by inverting
  such test statistics often depend substantially on how close to the boundary
  the true parameter is, leading to unreliable inference.
  The coverage of boundary points is addressed by existing
  methods, but the coverage of points near the boundary is not. We address both using
  a connection between boundary points and points where the Fisher information
  matrix is singular, which we call critical points. More specifically, we show
  many boundary points of interest are also critical points and use this to
  construct confidence regions that (i) have asymptotically correct uniform
  coverage probability, (ii) have empirical coverage close to nominal in
  simulations, and (iii) are straightforward to implement for many mixed models,
  including the ubiquitous linear mixed model. We know of no other confidence
  regions with properties (i) and (ii) in the settings we consider.
  Because not all critical points are boundary points, the
  proposed regions are useful in many settings where methods for inference on
  the boundary do not apply.

  To be more precise about the connections between boundary points, critical
  points, and mixed models, suppose a parameter $\theta \in \R{}$ scales a
  random effect with mean zero and unit variance in a mixed model, implying
  $\theta^2$ is a variance. For example, $\theta$ can be the coefficient of a
  random effect in a generalized linear mixed model. If
  the random effect has a distribution asymmetric around zero, inference on both
  the sign and magnitude of $\theta$ may be possible, in which case $\theta = 0$
  is not a boundary point. In other settings the sign is unidentifiable and
  inference on $\theta \geq 0$ and $\theta^2$ essentially equivalent;
  $\theta = 0$ is a boundary point. Either way, we will show that in quite
  general mixed models $\theta = 0$ is a critical point. More generally, when
  $\theta$ is a vector of parameters whose $j$th element $\theta_j$ is a scale
  parameter, $\theta$ is often a critical point if $\theta_{j} = 0$.
  Whether in a mixed model or not, inference near critical points is known to be
  difficult: the likelihood ratio test statistic and the maximum likelihood
  estimator behave quite differently than under classical conditions
  \citep{Rotnitzky.etal2000} and confidence regions obtained by inverting common
  test statistics such as the Wald, likelihood ratio, and score standardized by
  observed information have incorrect coverage probabilities \citep{Bottai2003}.
  By contrast, we show that, under regularity conditions, (i) the score test
  statistic standardized by expected Fisher information has a continuous
  extension at critical points and, when inverted, (ii) that test statistic
  gives a confidence region with asymptotically correct uniform coverage
  probability on compact sets. That is, the confidence region
  $\mathcal{R}_n(\alpha)$ based on $n$ observations with nominal level $(1 -
  \alpha) \in (0, 1)$ satisfies, for any compact subset $C$ of the parameter
  set,
  \begin{equation}\label{eq:unif:cover}
    \lim_{n\to \infty} \inf_{\theta \in C} \pr_\theta\{\theta \in
    \mathcal{R}_n(\alpha)\} = 1 - \alpha,
  \end{equation}
  where the subscript $\theta$ on $\pr$ indicates the data on which
  $\mathcal{R}_n$ is based have the distribution indexed by $\theta$.
  Importantly, $C$ can include neighborhoods of boundary and
  critical points. It is an immediate corollary that the test rejecting a null
  hypothesis $\theta = \theta_0$ when $\theta_0 \notin \mathcal{R}_n(\alpha)$
  has asymptotic size $\alpha$ for any $\theta_0$.  These results apply to but
  are not restricted to mixed models. Moreover, in contrast to many methods for
  testing variance parameters in mixed models, ours in general does not require
  the implementation of simulation algorithms or computing the maximum
  likelihood estimator, which can be complicated in non-linear mixed models.

  The connection between singular information and boundary settings has been
  noticed previously \citep{Cox.Hinkley2000, Chesher1984, Lee.Chesher1986}, but
  results similar to ours have only been obtained for settings with a single
  scalar parameter \citep{Bottai2003}. We recover those results as special
  cases, and under weaker conditions. Asymptotic properties of maximum
  likelihood estimators and likelihood ratio test statistics have been
  established for the special case where the rank of the Fisher information
  matrix is one less than full \citep{Rotnitzky.etal2000}, but confidence
  regions were not considered. Notably, our theory does not require the
  Fisher information to have a particular rank and, indeed, we will see that in
  mixed models the rank is often full minus the number of scale parameters equal
  to zero.

  We end this section with a simple example that illustrates how critical
  points often appear in mixed models. After the example, we give additional
  background and develop theory in Section \ref{sec:general}. In Section
  \ref{sec:mixed} we discuss the application to mixed models and verify the
  conditions of the theory from Section \ref{sec:general} in two such models.
  Section \ref{sec:sims} presents simulation results, Section \ref{sec:datex}
  contains a data example, and Section \ref{sec:final} concludes.
  \begin{exm} \label{ex:lmm}
  Suppose, for $i = 1, \dots, n$ and $j = 1, \dots, r$,
  \[
    Y_{i,j} = \theta W_i + E_{i,j},
  \]
  where $\theta \in [0, \infty)$ and all $W_i$ and $E_{i, j}$ are independent
  standard normal random variables. For example, $r$ can be the number of
  observations in a cluster, $n$ the number of clusters, and the random effect
  $W_i$ used to model heterogeneity between clusters or dependence between
  observations in the same cluster. The $Y_i = [Y_{i1}, \dots, Y_{ir}]^\tsp$, $i
  = 1, \dots, n$, are independent and multivariate normally distributed with
  mean zero and common covariance matrix $\Sigma(\theta) = \theta^2 1_r 1_r^\tsp
  + I_r$, where $1_r$ is an $r$-vector of ones and $I_r$ the $r\times r$
  identity matrix. With some algebra (Supplementary Material), one can show the log-likelihood for one
  observation $y_i\in \R{r}$ is
  \[
  \log f_\theta(y_i) =  -\frac{1}{2} \log(1 + \theta^2 r) - \frac{1}{2}
  \left\{y_i^\tsp y_i - (y_i^\tsp 1_r)^2 \theta^2 / (1 + r\theta^2) \right\}.
  \]
  Differentiating with respect to $\theta$ gives the score for one
  observation:
  \[
    s(\theta; y_i) = -\frac{r \theta}{1 + r \theta^2} + (y_i^\tsp 1_r)^2
   \frac{\theta}{(1 + r\theta^2)^2}.
  \]
  At $\theta = 0$, this score is zero for any $y_i\in \R{r}$ and, hence, the
  Fisher information is zero; that is, $\theta = 0$ is a critical point. There
  are no other critical points because the second term of $s(\theta; Y_i)$,
  $Y_i\sim f_\theta$, has positive variance when $\theta \neq 0$.

  Figure \ref{fig:ex_1} shows two (pseudo) randomly generated realizations of
  the log-likelihood in this example. For one dataset the critical point is a
  global maximizer and for the other a local minimizer. One can show that if the
  true $\theta$ is small, both types of outcomes have probability approximately
  $1 / 2$. In particular, the score always vanishes at $\theta = 0$ and the
  maximum likelihood estimator for $\theta$ is zero with probability
  approximately $1/2$. The maximum likelihood estimator's mass at zero gives
  some intuition for why confidence regions that directly or indirectly use
  asymptotic normality of that estimator can have poor coverage properties near
  the critical point (see \citeauthor{Rotnitzky.etal2000},
  \citeyear{Rotnitzky.etal2000}, and \citeauthor{Bottai2003},
  \citeyear{Bottai2003}, for details). In this example, the critical point is at
  the boundary since we assumed $\theta \geq 0$ for identifiability, but $\theta
  = 0$ would still be a critical point if the $W_i$ had an asymmetric
  distribution and the sign of $\theta$ were identifiable.

\begin{figure}[htb]
 \centering
 \includegraphics[width = 0.5 \textwidth]{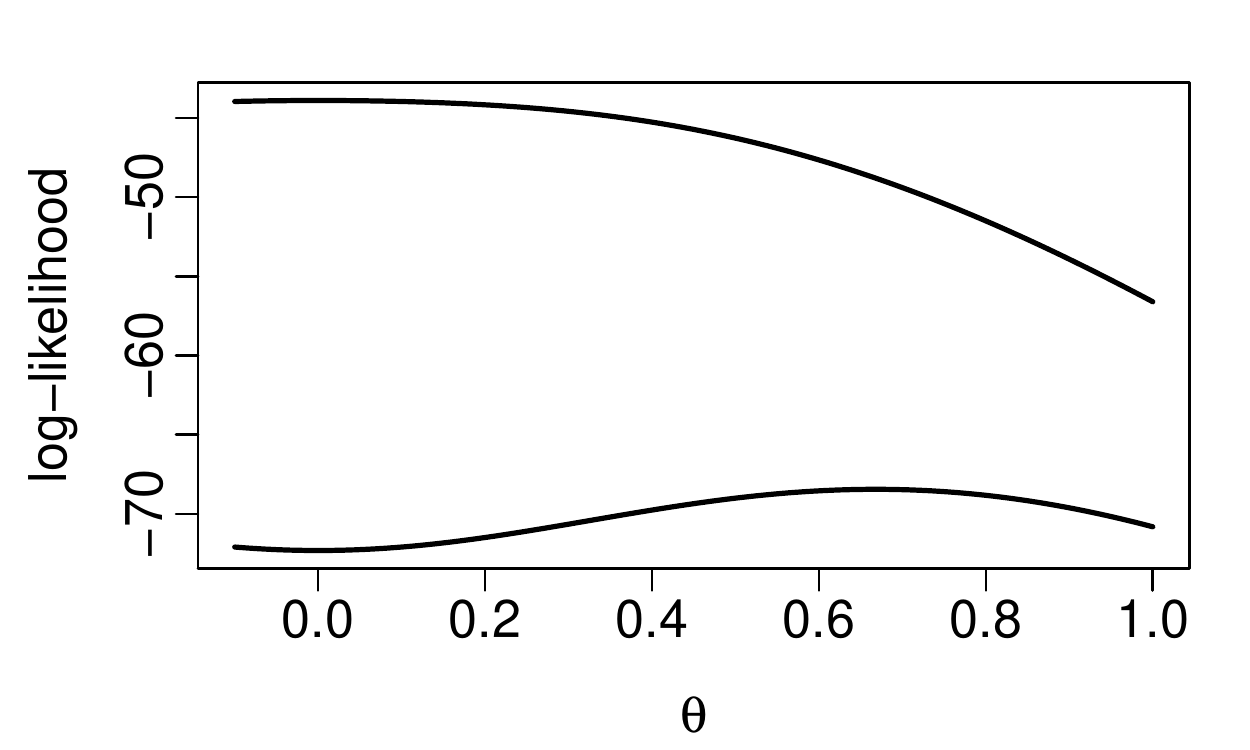}
 \caption{Log-likelihoods for two independent samples of $n = 100$
 independent observations each, generated with $\theta = 1 /
 \sqrt{10}$ and $r = 1$}
 \label{fig:ex_1}
\end{figure}
\end{exm}

\section{Inference near critical points} \label{sec:general}
\subsection{Definitions and assumptions}
  Suppose, independently for $i = 1, \dots, n$, $Y_i \in \R{r_i}$, $r_i \geq 1$,
  has density $f^i_\theta$ against a dominating measure $\gamma_i$, for $\theta$
  in a parameter set $\Theta \subseteq \R{d}$. For simplicity, we often write
  $f_\theta(y_i)$ in place of $f^i_\theta(y_i)$, where $y_i$ is an arbitrary
  realization of $Y_i$. Let $y^n = (y_1, \dots, y_n)$ be a realization of $Y^n =
  (Y_1, \dots, Y_n)$ and $\nabla$ denote the derivative operator with respect to
  $\theta$. Let also $\ell^i(\theta; y_i) = \log f_\theta(y_i)$, $\ell_n(\theta;
  y^n) = \sum_{i = 1}^n \ell^i(\theta; y_i)$, $s^i(\theta; y_i) = \nabla
  \ell^i(\theta; y_i)$, $s_n(\theta; y^n) = \nabla \ell_n(\theta; y^n)$,
  $\mathcal{I}^i(\theta) = \cov_\theta\{s^i(\theta; Y_i)\}$, and
  $\mathcal{I}_n(\theta) = \cov_\theta\{s_n(\theta; Y^n)\}$, where the subscript
  on $\cov$ indicates $Y^n = (Y_1, \dots, Y_n)$ has the distribution indexed by
  $\theta$; we may omit such subscripts when it has been explicitly stated which
  distribution the random variables have. Focus will be on inference near points
  where the Fisher information $\mathcal{I}_n$ is singular.

  \begin{defn}\label{def:crit}
    We say a point $\theta \in \Theta$ is critical if $\mathcal{I}_n(\theta)$ is
    singular and non-critical otherwise.
  \end{defn}

  Denote an arbitrary set of orthonormal eigenvectors of
  $\mathcal{I}_n(\theta)$ by $\{v_{\theta 1}^n, \dots, v_{\theta d}^n\}$. Then
  Definition \ref{def:crit} says, equivalently, that $\theta$ is a critical
  point if $s_n(\theta; Y^n)^\tsp v_{\theta j}^n$ is constant almost surely
  $\pr_\theta$ for at least one $j$. In our motivating examples, $\theta$ is the
  parameter vector in a mixed model but there are other potential applications
  for the theory in this section. For example, \citet[Section
  3.1.3]{Azzalini.Capitanio2014} consider a three-parameter multivariate
  skew-normal distribution with singular Fisher information.

  As mentioned in the Introduction, inverting common test statistics typically
  does not give confidence regions satisfying \eqref{eq:unif:cover} for subsets of
  the parameter set including critical points \citep{Bottai2003}. To address
  this, we will consider the score test statistic standardized by expected
  Fisher information:
  \begin{equation} \label{eq:teststat}
    T_n(\theta; y^n) = s_n(\theta; y^n)^\tsp
    \mathcal{I}_n(\theta)^{-1} s_n(\theta; y^n).
  \end{equation}
  The right-hand side of \eqref{eq:teststat} is undefined at critical $\theta$
  since $\mathcal{I}_n(\theta)$ is not invertible there. When
  $\theta$ is a scalar parameter, the
  following calculation can be formalized (proof of Theorem \ref{thm:cont}) to
  define a continuous extension at critical $\theta$:
  \begin{align*}
    \lim_{\delta \to 0}T_n(\theta + \delta;y^n) &= \lim_{\delta\to 0} [\{s_n(\theta + \delta; y^n) / \delta\}^2 \var_{\theta + \delta}\{s_n(\theta + \delta; Y^n) / \delta\}^{-1}]
    \\
    &= \{\nabla s_n(\theta; y^n)\}^2 \var_\theta\{  \nabla s_n(\theta; Y^n)\}^{-1},
  \end{align*}
  where the first line results from dividing and multiplying $T_n(\theta +
  \delta; y^n)$ by $\delta^2$, assuming $\theta + \delta$ is a non-critical
  point, and the second line uses the definition of derivative of $\theta\mapsto
  s(\theta; \cdot)$ at $\theta$ together with regularity conditions, ensuring
  among other things that $s_n(\theta; y^n) = 0$ at critical $\theta$. That is,
  at critical points the continuous extension of the test statistic for one
  parameter is based on the second derivative of the log-likelihood. Notably, if
  expected information is replaced by observed, which is not in general zero at
  critical points, the limit is typically zero.

  Now, our interest is twofold, namely conditions that ensure (i) $T_n$ has a
  continuous extension to critical points when $d \geq 1$ and (ii) confidence
  regions obtained by inverting that extension satisfy \eqref{eq:unif:cover}.
  Both (i) and (ii) are substantially more complicated when $d > 1$ than when
  $d = 1$ since the eigenvectors and the rank of the Fisher information
  become important.

  We will use the following assumptions.

  \begin{assmp}\label{assmp:support}
    For every $i = 1, 2, \dots$, the distribution $f_{\theta}(y_i) \gamma_i(\dd
    y_i)$ has the same null sets for every $\theta \in \Theta$.
  \end{assmp}
  In light of Assumption \ref{assmp:support}, we will often write "almost every
  $y_i$" without specifying the measure, implicitly referring to
  $f_\theta(y_i)\gamma_i(\dd y_i)$ for any $\theta \in \Theta$ or the corresponding
  product measure when such statements are about $y^n$. Notably, the null sets
  can be different for different $i$. Indeed, the $Y_i$ need not even take
  values in the same space.

  \begin{assmp} \label{assmp:deriv}

    For every $\theta' \in \Theta$, there exists an open ball $B = B(\theta')
    \subseteq \R{d}$ centered at $\theta'$ on which, for every $i = 1, 2, \dots$
    and almost every $y_i$, partial derivatives of $\theta \mapsto
    \ell^i(\theta; y_i)$ of an order $k = k(\theta') \geq 2$ exist and are
    jointly continuous in $(\theta, y_i)$. Moreover, there exists a $\delta =
    \delta(\theta') > 0$ such that every partial derivative of order at most $k$
    satisfies
    \begin{equation} \label{eq:deriv_cond}
      \sup_{i = 1, 2, \dots} \sup_{\theta \in B \cap \Theta, \tilde{\theta} \in
      B}\int \left\vert \frac{\partial^l}{\partial \theta_1^{l_1}\cdots \partial \theta_d^{l_d}}\ell^i(\theta; y_i) \bigg\rvert_{\theta  = \tilde{\theta}} \right\vert^{2 + \delta}
      f_\theta(y_i)\, \gamma_i(\dd y_i) < \infty,
    \end{equation}
    where $l_1, \dots, l_d$ are non-negative integers summing to $l \leq k$.
  \end{assmp}

  The ball $B$ in Assumption \ref{assmp:deriv} can include
  points not in $\Theta$. Then the assumption should be understood as saying
  there exist extensions of the partial derivatives in \eqref{eq:deriv_cond} to
  $B \times \mathcal{Y}_i$, for some set $\mathcal{Y}_i$ of full
  measure, satisfying the outlined conditions. We have sacrificed some
  generality for clarity in that, as will be clear later, the moment condition
  \eqref{eq:deriv_cond} can be weakened to apply to only certain partial
  derivatives of the order $k$. One consequence of Assumption \ref{assmp:deriv}
  is that $\E_\theta\{s_n(\theta; Y^n)\} = 0$ (Lemma D.2 in the Supplementary
  Material). Hence, $\mathcal{I}_n(\theta) v_{\theta
  j}^n = 0$ is equivalent to $s_n(\theta; y^n)^\tsp v_{\theta j}^n = 0$ for
  almost every $y^n$.

  For symmetric matrices $A$ and $B$ we write $A
  \succeq B$ or $B\preceq A$ if $A - B$ is positive semi-definite.

  \begin{assmp}\label{assmp:loew_order}
    There exist continuous $c_1, c_2:\Theta \to (0, \infty)$ such that, for
    every $i = 1, 2, \dots$ and $\theta \in \Theta$,
    \[
      c_1 (\theta) \mathcal{I}^1(\theta) \preceq \mathcal{I}^i(\theta) \preceq c_2(\theta)\mathcal{I}^1(\theta).
    \]
  \end{assmp}
  Assumption \ref{assmp:loew_order} holds with $c_1 = c_2 \equiv 1$ if the $Y_i$
  are identically distributed since in that case $\mathcal{I}^i =
  \mathcal{I}^1$. More generally, it controls the null space of $\mathcal{I}_n$:
  if $\mathcal{I}_n(\theta) v = 0$ for some $n$ and $v\in \R{d}$, then
  $\mathcal{I}^1(\theta) v = 0$ and, by Assumption \ref{assmp:loew_order},
  $\mathcal{I}^i(\theta) v = 0$ for every $i$. Thus, the critical points in
  Definition \ref{def:crit} do not depend on $n$. At non-critical points,
  Assumption \ref{assmp:loew_order} says, loosely speaking, that information in
  individual observations does not grow without bound or tend to zero. The
  choice of $\mathcal{I}^1$ is arbitrary in the sense that Assumption
  \ref{assmp:loew_order} holds as stated if and only if it holds with
  $\mathcal{I}^1$ replaced by any other $\mathcal{I}^i$.

  The next assumption requires some more notation. Let $\nabla_j^l$, $j = 1,
  \dots, d$, $l = 0, 1, \dots$, denote the $l$th order derivative operator with
  respect to $\theta_j$, with $\nabla_j = \nabla_j^1$. For
  example, $\nabla_j^l \ell_n(\tilde{\theta}; y^n) = \partial^l \ell_n(\theta;
  y^n) / \partial \theta_j^l \mid_{\theta = \tilde{\theta}}$. For every $\theta
  \in \Theta$ and $n = 1, 2, \dots$, let $k_j = k_j(\theta, n) \in \{1, 2,
  \dots\}$ ($j = 1, \dots, d$) be integers to be specified shortly and define,
  for $i = 1, \dots, n$,
  \[
    \tilde{s}_n^i(\theta; y_i) = [\{\nabla_{1}^{k_1 - 1} s^i(\theta; y_i)\}^\tsp
    v_{\theta 1}^n, \dots, \{\nabla_{d}^{k_d - 1} s^i(\theta; y_i)\}^\tsp
    v_{\theta d}^n]^\tsp \in \R{d}.
  \]
  Let also
  \[
      \tilde{s}_n(\theta; y^n) = \sum_{i = 1}^n \tilde{s}^i_n(\theta; y_i) ; \quad\tilde{\mathcal{I}}_n(\theta) = \cov_\theta\{\tilde{s}_n(\theta; Y^n)\}.
  \]
  We call
  $\tilde{s}_n(\theta; y^n)$ a modified score because it can understood as a
  two-step modification of $s_n(\theta; y^n)$: First $s_n(\theta; y^n)$ is
  rotated to $[v_{\theta 1}^n , \dots, v_{\theta d}^n]^\tsp s_n(\theta; y^n)$,
  in which the elements are uncorrelated linear combinations of first order
  derivatives of $\ell_n(\theta; y^n)$. Second, if $k_j > 1$ for some $j$, the
  linear combination in the $j$th element is replaced by a linear combination of
  higher-order partial derivatives. The idea to be formalized is to replace
  elements in $[v_{\theta 1}^n, \dots, v_{\theta d}^n]^\tsp s_n(\theta; y^n)$
  that are zero for almost every $y^n$ at a critical point by linear
  combinations of higher-order derivatives that are not.
  Note also $\tilde{s}_n(\theta; y_i)$ and $\tilde{\mathcal{I}}_n(\theta)$ depend on
  the $v_{\theta j}^n$. In particular, the set $\{v_{\theta 1}^n, \dots,
  v_{\theta d}^n\}$ is not uniquely determined in general and different choices
  lead to different $\tilde{s}_n(\theta; y^n)$. Our next assumption says three
  technical conditions must hold for at least one choice.

  \begin{assmp} \label{assmp:pd}
  For every $\theta \in \Theta$ and $n = 1, 2, \dots$, there exist integers
    $k_j = k_j(\theta, n) \in \{1, \dots, d\}$ $(j = 1, \dots, d)$ and a set of
    orthonormal eigenvectors $\{v_{\theta 1}^n, \dots, v_{\theta d}^n\}$ of
    $\mathcal{I}_n(\theta)$ such that:
    \begin{enumerate}
      \item[(i)~~] $\tilde{\mathcal{I}}_n(\theta)$ is positive definite;
      \item[(ii)~] $\{\nabla_{j}^{l} s_n(\tilde{\theta}; y^n)\}^\tsp v_{\theta
      j}^n = 0$, $l = 0, \dots, k_j - 2$, for every $\tilde{\theta}$ with
      $\tilde{\theta}_j = \theta_j$ and almost every $y^n$;
      \item[(iii)] $k_j(\theta, n)$ is upper bounded by the $k
      = k(\theta)$ in Assumption \ref{assmp:deriv} for every $j$, $\theta$, and $n$.
    \end{enumerate}
  \end{assmp}

  Condition (i) is essentially a more general, multidimensional version of the
  requirement by \citet{Bottai2003} that the second derivative of the
  log-likelihood has positive variance at critical points.
  Indeed, at non-critical $\theta$ (i) holds with $k_j = 1$ for
  all $j$ since $\tilde{s}_n(\theta; Y^n) = [v_{\theta 1}^n , \dots, v_{\theta
  d}^n]^\tsp s_n(\theta; Y^n)$ has diagonal covariance matrix with the
  eigenvalues of $\mathcal{I}_n(\theta)$ on the diagonal. Conversely, at
  critical $\theta$ at least one of those eigenvalues is zero and hence it must
  be that $k_j > 1$ for at least one $j$ in order for (i) to hold. For example,
  we will see mixed models where (i) holds at a critical $\theta$ with
  $v_{\theta 1}^n = e_1 = [1, 0, \dots, 0]^\tsp \in \R{d}$, $k_1= 2$, and $k_j =
  1$, $j = 2, \dots, d$. Then the first element of $\tilde{s}_n(\theta, y^n)$ is
  $\nabla_1^2 \ell(\theta; y^n)$ while the remaining elements are linear
  combinations of the elements of $s_n(\theta; y^n)$. While our theory only
  requires existence, implementing our method in practice can, at least in some
  models, require specification of a set of $k_j$ and numerical computation of a
  set of $v_{\theta j}$ satisfying (i); we discuss this further in Section
  \ref{sec:mixed}.

  Condition (ii) is vacuously satisfied at non-critical points since it only
  applies to $k_j \geq 2$. Consequently, Assumption \ref{assmp:pd} is weaker
  than assuming $\mathcal{I}_n(\theta)$ is positive definite for every $\theta$
  as is common in classical theory. To understand condition (ii) more generally,
  suppose $\theta$ is a critical point, $k_1 = 2$, and $v^n_{\theta 1}$ one of the
  eigenvectors of $\mathcal{I}_n(\theta)$ corresponding to the eigenvalue zero.
  As noted following Assumption \ref{assmp:deriv}, this implies $s_n(\theta;
  y^n)^\tsp v_{\theta 1}^n = 0$ for almost every $y^n$; condition (ii) says the
  same must hold if the score is evaluated at any other $\tilde{\theta}$ with
  first element $\tilde{\theta}_1 = \theta_1$. This motivates the following
  definition.

  \begin{defn}\label{def:crit_elem}
    We say $\theta_j$ is a critical element with corresponding critical
    (eigen-)vector $v$ if $\mathcal{I}_n(\tilde{\theta})v = 0$ at every
    $\tilde{\theta}$ with $\tilde{\theta_j} = \theta_{j}$.
  \end{defn}

  In general, an eigenvector of the Fisher information with vanishing eigenvalue
  need not be a critical vector. For example, if $v_1$ and $v_2$ are critical
  vectors with corresponding critical elements $\theta_1$ and $\theta_{2}$, then
  any linear combination of $v_1$ and $v_2$ is also an eigenvector with
  vanishing eigenvalue at every $\tilde{\theta}$ where both $\tilde{\theta}_1 =
  \theta_1$ and $\tilde{\theta}_2 = \theta_2$; but there is not a corresponding
  critical element.

  The possibility that $k_j > 2$ in condition (ii) means that, if condition (i)
  is not satisfied with $k_j \in \{1, 2\}$ for all $j$, then one may pass to
  higher order derivatives as long as the corresponding ones of lower order
  are zero for almost every $y^n$.

  Finally, we note Assumption \ref{assmp:pd} implies the rank of
  $\mathcal{I}_n(\theta)$ is $d$ minus the number of critical elements and that
  the assumption is sensitive to parameterization. For example,
  the model considered by \citet[Section 3.1]{Azzalini.Capitanio2014} does not
  satisfy Assumption \ref{assmp:pd} in the first parameterization discussed by
  the authors, but can be made to by a simple re-parameterization. Because the
  test statistic in \eqref{eq:teststat} is invariant under differentiable
  re-parameterizations with full rank Jacobian, it suffices to verify the
  conditions in one parameterization for the results to apply more generally.

  \begin{assmp} \label{assmp:dense}
    The non-critical points are dense in $\Theta$; that is, for any $\theta
    \in \Theta$, there exists a sequence $\{\theta_m\} \in \Theta$ of
    non-critical points tending to $\theta$.
  \end{assmp}

  In many settings the sets of critical elements are discrete subsets of $\R{}$.
  Then, if $\theta$ is a critical point with only critical element $\theta_j$,
  it is often possible to verify Assumption \ref{assmp:dense} with, for example,
  $\theta_m = \theta + m^{-1}e_j$.

  \subsection{Continuous extension}

  Our purpose in this section is to prove the following theorem.

  \begin{thm} \label{thm:cont}
    If Assumptions \ref{assmp:support}--\ref{assmp:deriv} and \ref{assmp:pd}--\ref{assmp:dense} hold, then for any
    $n\geq 1$ there is a set $\mathcal{Y}^n$ of full measure such that
    $T_n(\cdot; \cdot)$ has a continuous extension on $\Theta \times
    \mathcal{Y}^n$, and that extension is
    \[
      T_n(\theta; y^n) = \tilde{s}_n(\theta; y^n)
      \tilde{\mathcal{I}}_n(\theta)^{-1}  \tilde{s}_n(\theta;
      y^n).
    \]
  \end{thm}
  Here and in what follows, we use the same notation for the test statistic in
  \eqref{eq:teststat} and its continuous extension given by Theorem
  \ref{thm:cont}. The expressions in Theorem \ref{thm:cont} and
  \eqref{eq:teststat} agree at non-critical $\theta$ since $\tilde{s}_n(\theta;
  y^n) = [v_{\theta 1}^n, \dots, v_{\theta d}^n]^\tsp s_n(\theta; y^n)$ there,
  and the pre-mulitplication of any invertible matrix is canceled when
  standardizing by the covariance matrix. Observe the continuity in Theorem
  \ref{thm:cont} is jointly in the parameter and data.  Before giving a proof,
  we state and discuss some intermediate results used in that proof. The first
  is a lemma which, loosely speaking, says that if the non-critical points are
  dense, then it is enough to establish continuity along sequences of
  non-critical points for it to hold more generally. Proofs of formally stated
  results are in the Supplementary Material if not
  given here.

  \begin{lem} \label{lem:noncrit_to_crit}
    Suppose Assumption \ref{assmp:dense} holds and that, for every $\theta \in
    \Theta$ and $y^n$ in some $\mathcal{Y}^n \subseteq \R{r_1} \times \cdots
    \times \R{r_n}$, with $n \geq 1$ fixed, $\lim_{m\to \infty} T_n(\theta_m;
    y_m^n)$ exists and is the same for all sequences $\{\theta_m\} \in \Theta$
    of non-critical points tending to $\theta$ and sequences $\{y^n_m\} \in
    \mathcal{Y}^n$ tending to $y^n$; then $T_n(\cdot; \cdot)$ has a continuous
    extension on $\Theta \times \mathcal{Y}^n$.
  \end{lem}
  In order to use Lemma \ref{lem:noncrit_to_crit} to prove Theorem \ref{thm:cont},
  one must show that $T_n$ converges along sequences of non-critical points
  tending to critical points. Evaluating the score at such sequences gives a
  sequence of score vectors whose covariance matrices are non-singular but tend to
  a singular limit. The following lemma says those score vectors can be
  scaled to tend to a limit with positive definite covariance matrix.

  \begin{lem} \label{lem:scale_score_lim}
    Suppose Assumptions \ref{assmp:support}, \ref{assmp:deriv}, and
    \ref{assmp:pd} hold and let $\{\theta_m\} \in \Theta$ be a sequence of
    non-critical points tending to some $\theta \in \Theta$. Then there exist
    sequences of non-zero constants $\{a_{mj}\}$, $j = 1, \dots, d$, and a set
    $\mathcal{Y}^n$ of full measure such that, for any $\{y_m^n\} \in
    \mathcal{Y}^n$ tending to a $y^n\in \mathcal{Y}^n$ as $m\to\infty$,
    \[
    a_{mj} s_n(\theta_m; y^n_m)^\tsp v_{\theta j}^n \to \tilde{s}_{nj}(\theta;
    y^n),
    \]
    where $\tilde{s}_{nj}$ is the $j$th element of $\tilde{s}_n$ in Assumption
    \ref{assmp:pd}.
  \end{lem}

  \begin{proof}
    Let $\mathcal{Y}^n$ be the intersection of the sets of full measure in
    Assumptions \ref{assmp:deriv} and \ref{assmp:pd}. Since the limit point
    $\theta$ and $n$ are fixed, denote $v_j = v_{\theta j}^n$ for simplicity.
    For an arbitrary $j$ and all large enough $m$, condition (ii) of Assumption
    \ref{assmp:deriv} lets us apply Taylor's theorem with Lagrange-form
    remainder to the map $\theta_{mj} \mapsto v_{j}^\tsp s_n(\theta_m; y^n_m)$
    to get that $v_{j}^\tsp s_n(\theta_m; y^n_m)$ is equal to
    \[
      \sum_{l = 0}^{k_j - 2} \frac{(\theta_{mj} - \theta_{j})^l v_{j}^\tsp
      \{\nabla^l_{j} s_n(\theta_m^{(j)}; y^n_m) \}}{l!} + \frac{(\theta_{mj} -
      \theta_{j})^{k_j - 1} v_{j}^\tsp \{\nabla_j^{k_j - 1}
      s_n(\tilde{\theta}_m^{(j)}; y^n_m)\} }{(k_j - 1)!},
    \]
    where $\theta_m^{(j)}$ and $\tilde{\theta}_m^{(j)}$ are $\theta_m$ with
    $\theta_{mj}$ replaced by, respectively, $\theta_{j}$ and a point between
    $\theta_{mj}$ and $\theta_{j}$; and $k_j$ is selected in accordance with
    Assumption \ref{assmp:pd} at $\theta$. By condition (ii) of that assumption,
    the first $k_j-1$ terms in the last display vanish. Now set $a_{mj} = (k_j -
    1)! / (\theta_{mj} - \theta_{j})^{k_j - 1}$, which is always defined since
    $\theta_{m}$ is non-critical and hence $\theta_{mj}\neq \theta_j$ by
    condition (ii) of Assumption \ref{assmp:pd}. Then $a_{mj} v_j^\tsp
    s_n(\theta_m; y_m^n) = v_{j}^\tsp \{\nabla_j^{k_j - 1}
    s_n(\tilde{\theta}_m^{(j)}; y^n_m)\}$, which has the desired limit by
    continuity of partial derivatives given by Assumption \ref{assmp:deriv}.
  \end{proof}

  The importance of controlling the behavior of the eigenvectors of the Fisher
  information near critical points is highlighted by the proof of Lemma
  \ref{lem:scale_score_lim}: the first $k_j - 1$ terms in the Taylor expansions
  need not vanish if the eigenvectors depend on $\theta$ in a way violating
  condition (ii) of Assumption \ref{assmp:pd}. We are ready to prove Theorem
  \ref{thm:cont}.

\begin{proof}[Proof of Theorem \ref{thm:cont}]
  Fix $n$, $\theta \in \Theta$, and $y^n \in \mathcal{Y}^n$, where
  $\mathcal{Y}^n$ is the intersection of the sets of full measure given by
  Assumptions \ref{assmp:deriv} and \ref{assmp:pd}.
  Pick $V_n = [v_{\theta 1}^n, \dots, v_{\theta d}^n] \in \R{d\times d}$
  satisfying Assumption \ref{assmp:pd}. By Lemma \ref{lem:noncrit_to_crit}, it
  suffices to show
  \[
    \lim_{m\to \infty} s_n(\theta_m; y_m^n)^\tsp
    \mathcal{I}(\theta_m)^{-1}s_n(\theta_m; y_m^n) = \tilde{s}(\theta;
    y^n)^\tsp\tilde{\mathcal{I}}_n(\theta)^{-1}\tilde{s}(\theta; y^n)
  \]
  for an arbitrary sequence of non-critical points $\{\theta_m\}$ tending to
  $\theta$ and $\{y_m^n\}$ tending to $y^n$. We note it is enough to establish
  the limit for any one $V_n$ satisfying Assumption \ref{assmp:pd} since the
  sequence does not depend on the choice and, hence, neither can the limit. Let
  $A_m = \diag(a_{m1}, \dots, a_{md})$ be defined by the sequences of constants
  given by Lemma \ref{lem:scale_score_lim}. Then $A_m V_n^\tsp$ is invertible and
  consequently
  \[
    T_n(\theta_m; y^n_m) = \{A_m V_n^\tsp s_n(\theta_m; y^n_{m})\}^\tsp \{A_m
    V_n^\tsp \mathcal{I}_n(\theta_m)A_m V_n \}^{-1} A_m V_n^\tsp s_n(\theta_m;
    y^n_{m}).
  \]
  Lemma \ref{lem:scale_score_lim} says, for the $\tilde{s}_n$ in Assumption
  \ref{assmp:pd}, $ A_m V_n^\tsp s_n(\theta_m, y_{m}^n) \to \tilde{s}_n(\theta,
  y^n)$. Thus, we are done if we can show $A_m V_n^\tsp \mathcal{I}(\theta_m) V_n
  A_m  = \cov \{A_m V_n^\tsp s(\theta_m, Y_m^n)\} \to \tilde{\mathcal{I}}_n(\theta)$,
  where $Y^n_m = (Y_{m1}, \dots, Y_{mn})$ has the distribution indexed by
  $\theta_m$. To that end, note $\theta_m \to \theta$ implies $f^i_{\theta_m}
  \to f^i_{\theta}$ pointwise in $y_i$ for every $i$ by continuity implied by
  Assumption \ref{assmp:deriv}. Hence, the joint density for $Y^n_m$ tends
  pointwise to that of a $Y^n$ with distribution indexed by $\theta$. Thus,
  $Y^n_m \to Y^n$ in total variation by Scheffe's theorem \citep[Theorem
  16.12]{Billingsley1995}, and hence also in distribution. To show the desired
  convergence of covariance matrices, we may thus assume, by Skorokhod's
  representation theorem \citep[Theorem 6.7]{Billingsley1999}, that $Y^n_m \to
  Y^n$ almost surely. But then $A_m V_n^\tsp s_n(\theta_m; Y^n_m) \to
  \tilde{s}(\theta, Y^n)$ almost surely by Lemma \ref{lem:scale_score_lim}. Now
  convergence of the covariance matrices follows if the elements of the sequence
  $\{A_m V_n^\tsp s_n(\theta_m; Y^n_m) s_n(\theta_m; Y^n_m)^\tsp V_n A_m\}$ are
  uniformly integrable \citep[Theorem 25.12]{Billingsley1995}. To show they are,
  note the $j$th element of $A_m V_n^\tsp s_n(\theta_m, Y^n_m)$ is
  $\{\nabla_j^{k_j - 1} s_n(\tilde{\theta}_m^{(j)}; Y^n_m)\}^\tsp v_{\theta j}$,
  with $\tilde{\theta}_m^{(j)}$ selected as in Lemma \ref{lem:scale_score_lim}.
  Thus, the $(j, l)$th element of $A_m V_n^\tsp s_n(\theta_m; Y^n_m) s_n(\theta_m;
  Y^n_m)^\tsp V A_m$ is
  \begin{align*}
    \{\nabla_j^{k_j - 1}
    s_n(\tilde{\theta}_m^{(j)}; Y^n_m)\}^\tsp v_{\theta j} \{\nabla_l^{k_l - 1}
    s_n(\tilde{\theta}_m^{(l)}; Y^n_m)\}^\tsp v_{\theta l},
  \end{align*}
  which has uniformly bounded $(1 + \delta/2)$th moment by the Cauchy--Schwarz
  inequality and Assumption \ref{assmp:deriv}. From this the desired uniform
  integrability follows \citep[25.13]{Billingsley1995} and that completes the
  proof.
\end{proof}

\subsection{Asymptotic uniform coverage probability}

We now turn to confidence regions obtained by inverting the continuous extension
$T_n$. Specifically, for $\alpha \in (0, 1)$, define
\begin{equation} \label{eq:confreg}
  \mathcal{R}_n(\alpha) = \{\theta \in \Theta: T_n(\theta; Y^n) \leq q_{d, 1 -
  \alpha}\},
\end{equation}
where $q_{d, 1 - \alpha}$ is the $(1-\alpha)$th quantile of the chi-square
distribution with $d$ degrees of freedom. We have the following main result of
the section.

\begin{thm}\label{thm:main}
  Under Assumptions \ref{assmp:support}--\ref{assmp:dense}, the confidence
  region $\mathcal{R}_n(\alpha)$ in \eqref{eq:confreg} has asymptotically
  correct uniform coverage probability on compact sets; that is, it satisfies
  \eqref{eq:unif:cover}.
\end{thm}

We need some intermediate results before proving Theorem \ref{thm:main}. Our
strategy will be to prove that, for any compact $C \subseteq \Theta$ and $\alpha
\in (0, 1)$,
\begin{equation} \label{eq:unif:cover:exact}
  \lim_{n\to \infty} \sup_{\theta \in C}\left\vert \pr_\theta\{\theta \in \mathcal{R}_n(\alpha)\} - (1 - \alpha)\right\vert = 0,
\end{equation}
which implies \eqref{eq:unif:cover}. The following two lemmas let us focus on
convergence along sequence of non-critical points as in the previous section,
but now taking the stochastic properties of the data into account.

\begin{lem}\label{lem:cover_to_seq}
  Equation \eqref{eq:unif:cover:exact} holds for every compact
  $C\subseteq\Theta$ if and only if, for every convergent sequence
  $\{\theta_n\} \in \Theta$ as $n\to \infty$,
  \begin{equation}\label{eq:cont:conv:dist}
    T_n(\theta_n; Y^n_n) \rightsquigarrow \chi^2_d,
  \end{equation}
  where $Y_n^n = (Y_{n 1}, \dots, Y_{n n})$ has the distribution indexed by
  $\theta_n$.
\end{lem}

The proof of Lemma \ref{lem:cover_to_seq} essentially amounts to showing
continuous convergence is equivalent to uniform convergence on compact sets, the
function of interest being $\theta \mapsto \pr_\theta\{\theta \in
\mathcal{R}_n(\alpha)\}$. Lemma \ref{lem:cover_to_seq}
suggests, roughly speaking, that to get reliable confidence regions the
distribution of the test statistic should be the same regardless of how close to
the critical point the true parameter is.

The next lemma says it suffices to consider sequences of non-critical points.
Specifically, we need not treat the asymptotic distribution at critical
points separately.
\begin{lem}\label{lem:noncrit_to_crit_2}
  If Assumptions \ref{assmp:support}--\ref{assmp:dense} hold and
  \eqref{eq:cont:conv:dist} holds for every convergent sequence $\{\theta_n\}
  \in \Theta$ of non-critical points, then \eqref{eq:cont:conv:dist} holds for
  any convergent sequence in $\Theta$.
\end{lem}

\begin{proof}
  Let $F_n$ denote the cumulative distribution function of $T_n(\theta_n;
  Y_n^n)$ and let $F$ denote that of $\chi^2_d$. Assumption \ref{assmp:dense}
  says that, for every fixed $n$, we can pick a sequence of non-critical points
  $\{\theta_n^m\}$ tending to $\theta_n$ as $m\to\infty$ with $n$ fixed.
  Assumptions \ref{assmp:support}--\ref{assmp:dense} ensure Theorem
  \ref{thm:cont} holds and this implies, essentially by Slutsky and continuous
  mapping theorems (Lemma D.1 in the Supplementary Material),
  \[
  T_n(\theta_n^m; Y^n_m) \rightsquigarrow T_n(\theta_n; Y_n^n), ~~m\to \infty,
  \]
  where $Y^n_m = (Y_{m 1}, \dots, Y_{m n})$ has the distribution indexed by
  $\theta_n^m$. Thus, the corresponding cumulative distribution functions
  $\{F_n^m\}$ tend to $F_n$ at every point of continuity of $F_n$ as
  $m\to\infty$ with $n$ fixed. Let $D_n$ be the set of discontinuities of $F_n$
  and $D = \cup_{n} D_n$. Since any cumulative distribution function has at most
  countably many discontinuities, $D$ is countable as a countable union of
  countable sets. Now for any $t \in \R{}\setminus D$, we can pick, for every
  $n$, an $m = m(n)$ large enough that $\Vert \theta^m_n - \theta_n\Vert \leq
  1/n$ and $\vert F^{m(n)}_n(t) - F_n(t)\vert \leq 1 / n$. We then have by the
  triangle inequality,
  \[
    \vert F_n(t) - F(t)\vert \leq 1 / n + \vert F_n^{m(n)}(t) - F(t)\vert,
  \]
  which tends to zero as $n\to\infty$ by the assumption that
  \eqref{eq:cont:conv:dist} holds along sequences of non-critical points since
  $F$ is continuous; in particular, $t$ is a point of continuity of $F$. The
  proof is completed by observing that, since $D$ is countable, $\R{}\setminus
  D$ is dense in $\R{}$ and hence the convergence in fact holds at every $t \in
  \R{}$ \citep[Proposition 2, Chapter 14]{Fristedt.Gray2013}.
\end{proof}

We are ready to prove Theorem \ref{thm:main}

\begin{proof}[Proof of Theorem \ref{thm:main}]
  By Lemma \ref{lem:noncrit_to_crit_2}, it suffices to consider an arbitrary
  sequence $\{\theta_n\}$ of non-critical points tending to a $\theta \in
  \Theta$. Let $U_{n} = A_n V_n^\tsp s_n(\theta_n; Y_n^n)$, where $V_n =
  [v^n_{\theta 1}, \dots, v_{\theta d}^n] \in \R{d \times d}$ satisfies
  Assumption \ref{assmp:pd} at $\theta$ for every $n$, $A_n = \diag(a_{n1},
  \dots, a_{nd})$ is a scaling matrix defined by the $\{a_{nj}\}$ given by Lemma
  \ref{lem:scale_score_lim} (with the index $m = n$), and $Y_n^n = (Y_{n1},
  \dots, Y_{nn})$ has the distribution indexed by $\theta_n$. By re-ordering the
  elements of $\theta$ if necessary, we may partition $V_n = [V_{n1}, V_{n2}]$
  and have that the columns of $V_{n1}$ are the critical vectors at $\theta$. As
  noted following Assumption \ref{assmp:loew_order}, these actually do not
  depend on $n$, but $V_{n 2}$ may. With these definitions,
  $T_n(\theta_n; Y^n_n) = U_n^\tsp \cov(U_n)^{-1}U_n$ and, hence, the continuous
  mapping theorem  \citep[Theorem 2.7]{Billingsley1999} says it suffices to show
  $\cov(U_n)^{-1/2}U_n \rightsquigarrow \rN(0, I_d)$.

  Let $U_{nt} = t^\tsp U_n$ for an arbitrary $t\in \R{d}$ with $\Vert t\Vert =
  1$. We start by verifying Lyapunov's conditions \citep[Theorem
  27.3]{Billingsley1995} for $U_{nt} / \sigma_{nt} \rightsquigarrow \rN(0, 1)$,
  where $\sigma^2_{nt} = \var(U_{nt})$. First, $\E(U_{nt}) = 0$ since
  $\E\{s_n(\theta_n; Y^n_n)\} = 0$. Thus,
  it suffices to show $\sigma^2_{nt} \geq \epsilon n$ for some $\epsilon > 0$
  and $\E(\vert U_{nti}\vert^{2 + \delta}) \leq M$ for some $\delta > 0$ and $M
  < \infty$, where $U_{nti} = t^\tsp A_n V_n^\tsp s^i(\theta_n; Y_{ni})$ is the
  $i$th summand of $U_{nt}$. For the latter we have, since $\Vert t\Vert = \Vert
  v_{nj}\Vert = 1$, using the triangle and generalized mean inequalities, with
  $\tilde{\theta}^{(j)}$ as in the proof of Lemma \ref{lem:scale_score_lim},
  \begin{align*}
    \vert U_{nti}\vert^{2 + \delta}  \leq \left(\sum_{j = 1}^d \vert v_{n
    j}^\tsp \nabla_{j}^{k_j - 1 } s^i(\tilde{\theta}^{(j)}_n;
    Y_{ni})\vert\right)^{2 + \delta} \leq 4^{1 + \delta} \sum_{j = 1}^d \sum_{l
    = 1}^d \vert \nabla^{k_j - 1}_j \nabla_l
    \ell^i(\tilde{\theta}_n^{(j)}; Y_{ni})\vert^{2 + \delta},
  \end{align*}
  whose expectation is less than some $M < \infty$ for all large enough $n$ by
  Assumptions \ref{assmp:deriv} and \ref{assmp:pd}.

  Next, by Assumption \ref{assmp:loew_order},
  \[
  \sigma^2_{nt} \geq n c_1(\theta_n) t^\tsp A_n
  V_n^\tsp \mathcal{I}^1(\theta_n) V_n A_n t \geq n c_1(\theta_n)\emin\{A_n
  V_n^\tsp \mathcal{I}^1(\theta_n) V_n A_n\},
  \]
  where $\emin(\cdot)$ is the smallest eigenvalue.
  Since $c_1$ is continuous and positive at $\theta$, the right-hand side is
  greater than $ \epsilon n$ for some $\epsilon > 0$ if $\liminf_{n\to \infty}
  \emin\{A_n V_n^\tsp \mathcal{I}^1(\theta_n) V_n A_n\} > 0$. We will
  prove this by contradiction. To that end, suppose $\liminf_{n\to
  \infty} \emin\{A_n V_n^\tsp \mathcal{I}^1(\theta_n) V_n A_n\} = 0$ and
  extract a subsequence tending to zero. Then, by compactness of the set of
  semi-orthogonal matrices and the Bolzano--Weierstrass property
  \citep[Theorem 0.25]{Folland2007}, there is a further subsequence along
  which $V_{n2}$ tends to a semi-orthogonal $V_2$. Along this
  subsequence, by arguments almost identical to those in the proof of Theorem
  \ref{thm:cont}, $A_n V_n^\tsp s_1(\theta_n; Y_{n1})$ tends in
  distribution to a random vector with positive definite covariance matrix,
  and $A_n V_n^\tsp \mathcal{I}^1(\theta_n) V_n A_n$ tends to that
  covariance matrix. Thus, along the subsequence, by Weyl's inequalities
  \citep[Section III.2]{Bhatia2012}, $\emin\{A_n V_n^\tsp
  \mathcal{I}^1(\theta_n) V_n A_n\}$ tends to a strictly positive number,
  which is the desired contradiction.

  We have proven $U_{nt} / \sigma_{nt} \rightsquigarrow \rN(0, 1)$, from which
  $\cov(U_n)^{-1/2}U_n \rightsquigarrow \rN(0, I_d)$ follows \citep[Lemma
  2.1]{Biscio.etal2018} if
  \[
    0 < \liminf_{n\to\infty} \emin\{n^{-1}\cov(U_{n})\} \leq \limsup_{n
    \to\infty}\emax\{n^{-1}\cov(U_{n})\} < \infty,
  \]
  where $\emax(\cdot)$ is the maximum eigenvalue.
  For the first inequality, Assumption \ref{assmp:loew_order} gives $n^{-1}
  \cov(U_n) \succeq c_1(\theta_n) A_n V_n^\tsp \mathcal{I}^1(\theta_n) V_n A_n$,
  and we have already shown the right-hand side has smallest eigenvalue
  asymptotically bounded away from zero. A similar argument, using that
  Assumption \ref{assmp:loew_order} implies $\emax\{n^{-1} \cov (U_n)\} \leq
  c_2(\theta_n) \emax\{A_n V_n^\tsp \mathcal{I}^1(\theta_n)V_n A_n\}$,
  establishes the upper bound and this completes the proof.
\end{proof}

Theorems \ref{thm:cont} and \ref{thm:main} have the following corollary which
recovers a result of \citet{Bottai2003} but with several conditions weakened.
The proof is a straightforward verification of Assumptions
\ref{assmp:deriv}--\ref{assmp:dense} and hence omitted.

\begin{corol} \label{corol:uni_param}
  If $d = 1$; the $Y_i$ are identically distributed; Assumption
  \ref{assmp:support} holds; for every $\theta' \in \Theta$
  there exists an open ball $B = B(\theta') \subseteq \R{}$ centered at $\theta'$ such that
  (i) $\nabla^2 \ell^1(\theta; y_1)$ exists on $B$ for almost every $y_1$ and is
  continuous in $(\theta, y_1)$, and (ii) it holds that
  \[
  \sup_{\theta \in B \cap \Theta, \tilde{\theta} \in B}\int \vert
\nabla^2 \ell^1(\tilde{\theta}; y_1)\vert^{2 + \delta}f_\theta(y_1)\,\gamma_1(\dd y_1) < \infty
  \]
  for some $\delta = \delta(\theta') > 0$; and $\var_{\theta} \{\nabla^2 \ell^1(\theta; Y_1)\} > 0$ for every critical $\theta$; then the conclusions of Theorems
  \ref{thm:cont} and \ref{thm:main} hold.
\end{corol}

We end this section by illustrating the introduced ideas in the example from
the introduction. More complicated mixed models are considered in the next section.

\begin{customexm}{1}[continued]
  Recall that the $Y_i \in \R{r}$, $i = 1, \dots, n$, are independent and
  multivariate normally distributed with mean $0$ and covariance matrix
  $\Sigma(\theta) = \theta^21_r1_r^\tsp + I_r$, and that the score
  for one observation is
  \[
    s^i(\theta; y_i) = -\frac{r \theta}{1 + r \theta^2} +
    (y_i^\tsp 1_r)^2 \frac{\theta}{(1 + r\theta^2)^2} = \frac{\theta}{1 + r
    \theta^2}\left(-r + \frac{(y_i^\tsp 1_r)^2}{1 + r \theta^2} \right).
  \]
  Since $Y_i$ has positive variance for all $\theta$, the score is equal to zero
  for almost every $y_i$ if and only if $\theta = 0$, verifying Assumption
  \ref{assmp:dense}. Assumption \ref{assmp:support} holds since $f_\theta(y_i) >
  0$ for all $y_i \in \R{r}$, with $\gamma_i \equiv \gamma$ being Lebesgue
  measure on $\R{r}$. To verify Assumption \ref{assmp:pd}, observe that at
  $\theta = 0$ the second derivative of the log-likelihood is
  \[
    \lim_{\theta \to 0}\frac{s^i(\theta; y_i) - s^i(0; y_i)}{\theta}
    =\lim_{\theta \to 0} \left[\frac{1}{1 + r \theta^2}\left(-r + \frac{(y_i^\tsp
    1_r)^2}{1 + r \theta^2} \right)\right] = -r + (y_i^\tsp 1_r)^2,
  \]
  which has variance bounded away from zero when evaluated at $y_i = Y_i$. Thus,
  Assumption \ref{assmp:pd} holds at $\theta = 0$ with $k_1 = 2$ and at $\theta
  \neq 0$ with $k_1 = 1$. It is straightforward to verify Assumptions
  \ref{assmp:deriv} and \ref{assmp:loew_order}, and hence conclude Theorems
  \ref{thm:cont} and \ref{thm:main} apply. For additional insight, we also
  provide a more direct argument for why the score test standardized by expected
  information works in this example while other common test statistics do not.
  At $\theta \neq 0$ (see the Supplementary Material
  for details),
  \[
    T_n(\theta; Y^n) = \frac{1}{2r n} \left\{-rn +  \sum_{i =
    1}^n(Y_i^\tsp 1_r)^2 / (1 + r \theta^2) \right\}^2 \sim \frac{(-rn +  r
    \chi^2_n )^2}{2r n}.
  \]
  It is immediate from the middle expression that $T_n(\cdot;\cdot)$ has a
  continuous extension on $[0, \infty) \times \R{nr}$. Essentially,
  standardizing by expected information cancels the leading factor $\theta / (1
  + r \theta^2)$ in the expression for $s^i(\theta; y_i)$, which was the reason
  for the singularity at $\theta = 0$. Notably, this cancellation does not
  happen if one instead standardizes by observed information. The last
  expression shows that, in this example, the distribution of the proposed test
  statistic is in fact independent of the parameter, and hence it is almost
  immediate that \eqref{eq:cont:conv:dist} and, hence, Theorem \ref{thm:main}
  hold; writing $\chi^2_n$ as the sum of $n$ independent $\chi^2_1$ and an
  appeal to the classical central limit theorem is all that is needed. To
  emphasize the fact that other common test statistics do not enjoy the same
  asymptotic properties, we show in Theorem C.1 of the Supplementary Material
  that the asymptotic distribution of the score test statistic standardized by
  observed information evaluated at the true $\theta_n$, is different depending
  on how $\{\theta_n\}$ tends to $0$.
\end{customexm}

\section{Inference near critical points in mixed models} \label{sec:mixed}
\subsection{Scale parameters at zero}
Suppose momentarily the $Y_i$ are identically distributed; that
is, they are independent copies of some $Y \in \R{r}$. Then,
$\mathcal{I}_n(\theta) = n \mathcal{I}_1(\theta)$ and we write
$\mathcal{I}(\theta) = \mathcal{I}_1(\theta)$. Similarly, we drop the
observation index on $y \in \R{r}$, $s(\theta; y) = \nabla \log
f_\theta(y)$, and $\gamma$. Suppose further $Y$ has conditional density
$f_{\theta}(y\mid w)$ against $\gamma$ given a vector of random effects $W \in
\R{q}$ whose elements are independent with mean zero and unit variance. Denote
the distribution of $W$ by $\nu$. We partition the parameter vector as $\theta =
(\lambda, \psi) \in \R{d_1}\times \R{d_2}$ and assume, for some parameter matrix
$\Lambda = \Lambda(\lambda)\in \R{q\times q}$ and known $h:\R{r} \times
\R{d_2} \times \R{q} \to [0, \infty)$,
\[
  f_\theta(y\mid w) = h(y, \psi, \Lambda(\lambda) w).
\]
We call $\lambda$ a scale parameter since it determines the matrix
$\Lambda(\lambda)$ scaling $W$. The vector $\psi$ includes any other parameters
in the distribution of $Y\mid W$. With these assumptions, the marginal
distribution of $Y$ has density against $\gamma$ given by
\begin{equation}\label{eq:model}
  f_\theta(y) = \int f_\theta(y\mid w)\,\nu(\dd w) = \int h(y,
  \psi, \Lambda w)\, \nu(\dd w).
\end{equation}
For example, a generalized linear mixed model with linear predictor $X \psi +
Z U$, where $U = \Lambda  W \sim \rN(0, \Lambda \Lambda^\tsp)$, and $X \in
\R{r\times d_2}$ and $Z \in \R{r\times q}$ are design matrices, satisfies
\eqref{eq:model} for an appropriate choice of $h$.

In practice it is often possible to define $\psi$, $\lambda$, $\Lambda$ and $h$
so that any critical points are so because the leading $d_1\times d_1$ block of
$\mathcal{I}(\theta)$ is singular. That is, all critical points are due to
linear combinations of the scores for the scale parameters being equal to zero
almost surely. To investigate when the latter can happen, let $\nabla_{(3)} h(y,
\psi, \Lambda w)$ denote the gradient of $h(y, \psi, \cdot)$ evaluated at
$\Lambda w$, the subscript (3) here indicating gradient with respect to the
third argument vector. Then, assuming the derivatives exist and can be moved inside the
integral, the $j$th element of $s(\theta; y)$, $j = 1, \dots, d_1$, is
\begin{equation}\label{eq:mm_score}
  \frac{1}{f_\theta(y)}\int \nabla_j
  f_\theta(y\mid w)\,\nu(\dd w) =  \frac{1}{f_\theta(y)}\int \{\nabla_{(3)} h(y, \psi,
  \Lambda w)\}^\tsp \{\nabla_j \Lambda(\lambda)\} w
  \,\nu(\dd w),
\end{equation}
where $\nabla_j$ acts elementwise on matrices. The following result gives a
sufficient condition for linear combinations of these scores to be equal to
zero for almost every $y$.

\begin{prop} \label{prop:crit}
  If \eqref{eq:mm_score} holds at $\theta \in \Theta$ and for a $v =
  [v_1, \dots, v_{d_1}, 0, \dots, 0]^\tsp \in \R{d}$ it holds that
  \[
    \sum_{j = 1}^{d_1} v_j \{\nabla_j \Lambda(\lambda)\} W~~\text{and}~~ \Lambda(\lambda)W
  \]
  are independent under $\theta$; then $\mathcal{I}(\theta)v = 0$.
\end{prop}

\begin{proof}
  It suffices to show that $v^\tsp s(\theta; y) = 0$ for almost every $y$. We have
  \[
    v^\tsp s(\theta; y) = \frac{1}{f_\theta(y)} \int \{\nabla_{(3)} h(y, \psi,
    \Lambda w)\}^\tsp \sum_{j = 1}^{d_1} v_j \{\nabla_j \Lambda(\lambda)\} w \,\nu(\dd w).
  \]
  By the assumptions, the integral is the expectation of the inner product of
  two independent random vectors. Thus, since (measurable) functions of
  independent random variables are independent, the integral in the last display is
  \[
  \int \{\nabla_3 h(y, \psi,
  \Lambda w)\}^\tsp \,\nu(\dd w)  \int \sum_{j = 1}^{d_1} v_j \{\nabla_j \Lambda(\lambda)\} w \,\nu(\dd w),
  \]
  which is equal to zero since $\int w\, \nu(\dd w) = 0$, and that completes
  the proof.
\end{proof}

Proposition \ref{prop:crit} can help identify critical points. In fact, even
though the condition that $\sum_{j = 1}^{d_1} v_j \{\nabla_j \Lambda(\lambda)\}
W$ and $\Lambda W$ are independent random variables is not necessary, we will
see that in practice it often identifies all critical points.

In what follows, to facilitate further analysis, we assume $\Lambda$ is
diagonal. Specifically, we assume every diagonal element of $\Lambda$ is one of the
$\lambda_j$. Then the $\lambda_j$ are scale parameters in the usual sense and
\begin{equation} \label{eq:indep_scale}
  \Lambda(\lambda) = \diag(\lambda_{(1)}, \dots, \lambda_{(q)}),
\end{equation}
where $\lambda_{(l)}$ means the $\lambda_j$ ($j = 1, \dots
,d_1$) in the $l$th diagonal element of $\Lambda(\lambda)$ ($l
= 1, \dots, q$). For example, if $q = 3$
one possibility is that $\Lambda(\lambda) = \diag(\lambda_1, \lambda_2,
\lambda_1)$ so that the first and third element of $W$ have the same scale
parameter; then $d_1 = 2$, $\lambda_{(1)} = \lambda_1$, $\lambda_{(2)} =
\lambda_2$, and $\lambda_{(3)} = \lambda_1$. Assuming \eqref{eq:indep_scale} is common in
practice and is less restrictive than it may first seem: it allows for the
possibility that $w$ affects the conditional density $f_\theta(y\mid w)$ through
$H \Lambda w$ for some $H\in \R{q\times q}$, which may depend on $\psi$. Then $U
= H \Lambda W$ is a vector of dependent random effects whose scales are
determined by $\lambda$ and whose dependence is determined by $H$. In
particular, by letting $H$ be an orthogonal matrix the $\lambda_j^2$ are the
eigenvalues of $\cov_\theta(U) = H \Lambda(\lambda)^2 H^\tsp$.

With \eqref{eq:indep_scale}, Proposition \ref{prop:crit} has the following
corollary which says standard basis vectors are critical vectors for scale
parameters at zero.

\begin{corol} \label{corol:indep_crit}
  If \eqref{eq:mm_score} holds at $\theta \in \Theta$, $\Lambda(\lambda)$
  satisfies \eqref{eq:indep_scale}, and $\theta_j = \lambda_j = 0$; then
  $\mathcal{I}(\theta) e_j = 0$, where $e_j$ is the $j$th standard basis vector
  in $\R{d}$.
\end{corol}

\begin{proof}
  With \eqref{eq:indep_scale}, $\nabla_j \Lambda(\lambda)$, $j
  = 1, \dots, d_1$, is a diagonal matrix whose $l$th diagonal element ($l = 1,
  \dots, q$) is 1 if the $l$th diagonal element of $\Lambda(\lambda)$ is
  $\lambda_j$, and zero otherwise. Thus, at $\theta$ such that $\lambda_j = 0$,
  $\{\nabla_j \Lambda(\lambda)\} W$ is a function of the elements of $W$ scaled
  by $\lambda_j$, and $\Lambda(\lambda) W$ is a function of the elements of $W$
  not scaled by $\lambda_j$. Thus, Proposition \ref{prop:crit} is satisfied with
  $v = e_j$.
\end{proof}

We illustrate the wide applicability of Corollary \ref{corol:indep_crit} using
another example.

\begin{exm}[Generalized linear mixed model] \label{ex:glmm}
  Let $X \in \R{r\times d_2}$ and $Z\in \R{r \times q}$ be design matrices and
  suppose
  \[
    f_\theta(y\mid w) = h(y, \psi, \Lambda w) = \exp\left\{y^\tsp (X\psi + Z \Lambda w) -
    c(X\psi + Z \Lambda w)\right\},
  \]
  where $c: \R{r} \to \R{}$ is the sum of the cumulant functions for the $r$
  responses \citep[see e.g.][for definitions]{McCulloch.etal2008}. For example,
  the $j$th element of $\nabla c(X\psi + Z \Lambda w)$, the gradient of $c$
  evaluated at the linear predictor, is the conditional mean of the $j$th
  response given $w$. Assume \eqref{eq:indep_scale} and, for $j = 1, \dots,
  d_1$, let $[j]$ denote the set of $k \in \{1, \dots, q\}$ such that
  $\Lambda_{kk} = \lambda_j$. Then,
  \[
   \nabla_j Z \Lambda(\lambda) w  = \sum_{k \in [j]}Z^k w_k,
  \]
  where $Z^k$ is the $k$th column of $Z$. Hence, assuming differentiation under
  the integral is permissible,
  \begin{align*}
    s_j(\theta; y) &= \frac{1}{f_\theta(y)}\int f_\theta(y\mid w) \{y -
    \nabla c(X\psi + Z\Lambda w)\}^\tsp\sum_{k \in [j]} Z^k w_k \,\nu(\dd w).
  \end{align*}
  Corollary \ref{corol:indep_crit} suggests $s_j(\theta; y)$ is zero for almost
  every $y$ at $\theta$ such that $\lambda_j = 0$, which can here be
  verified directly: the sum in the integral is a function of the $w_k$ scaled
  by $\lambda_j$, while the remaining part of the integrand is a function of the
  $w_k$ not scaled by $\lambda_j$. Hence, the integral is
  \[
    \int f_\psi(y\mid \Lambda w) [y - \nabla c(X\psi + Z\Lambda
    w)]^\tsp\,\nu(\dd w) \int \sum_{k \in [j]} Z^k w_k \,\nu(\dd w),
  \]
  which is equal to zero for every $y$ since the second integral is
  the expectation of a linear combination of the $W_k$, which have mean zero.
\end{exm}

  It is clear from Example \ref{ex:glmm} that critical points occur in many
  mixed models and that standard basis vectors are often critical vectors. When
  all critical vectors are standard basis vectors, the calculations required to
  verify Assumption \ref{assmp:pd} are simplified. The following result
  illustrates this point and will be useful in examples.

  \begin{prop}\label{prop:assmp4}
    If Assumptions \ref{assmp:support}--\ref{assmp:deriv} hold, (i) $Y_1, \dots,
    Y_n$ are independent copies of $Y\in \R{r}$, (ii) for every $\theta \in
    \Theta$ the null space of $\mathcal{I}(\theta)$ is spanned by $\{e_j, j:
    \lambda_j = 0\}$, and (iii) the random vector
    \[
      \bar{s}(\theta; Y) = \left[\nabla^{k_1}_1 \log f_\theta(Y), \dots,
      \nabla^{k_d}_d \log f_\theta(Y)\right]^\tsp, ~~ Y \sim f_\theta,
    \]
    where $k_j = k_j(\theta) = 2$ if $\theta_j = \lambda_j = 0$ and $k_j = 1$ otherwise, has
    positive definite covariance matrix with finite entries at every $\theta \in
    \Theta$; then Assumption \ref{assmp:pd} holds.
  \end{prop}

  \begin{proof}[Proof of Proposition \ref{prop:assmp4}]
    Let $\{v_{\theta 1 }, \dots, v_{\theta d}\}$ be orthonormal eigenvectors of
    $\mathcal{I}(\theta)$ and $V = [v_{\theta 1 }, \dots, v_{\theta d}]$.
    Suppose the rank of $\mathcal{I}(\theta)$ is $d - d_*$, $0\leq d_* \leq d$.
    By condition (i), upon re-ordering $\theta$ if necessary, we may assume
    $v_{\theta j} = e_j$, $j = 1, \dots, d_*$. In the definition of
    $\tilde{s}_n(\theta; y^n)$, set $k_1 = \cdots = k_{d*} = 2$ and $k_{d* + 1}
    = \cdots k_d = 1$. Then, since $v_{\theta j}$ has zeros in the first $d_*$
    entries for $j \geq d_* + 1$ by orthogonality, $\tilde{s}_n(\theta; Y^n) =
    V^\tsp \sum_{i = 1}^n \bar{s}(\theta; Y_i)$. Thus,
    $\tilde{\mathcal{I}}_n(\theta) = n V^\tsp \cov_\theta\{\bar{s}(\theta;
    Y)\}V$, which is positive definite if and only if
    $\cov_\theta\{\bar{s}(\theta; Y)\}$ is; this shows condition (i) of
    Assumption \ref{assmp:pd} holds.

    Condition (ii) of Assumption \ref{assmp:pd} holds because $s_n(\tilde{\theta};
    y^n)^\tsp e_j = 0$ for almost every $y^n$ is equivalent to
    $\mathcal{I}_n(\tilde{\theta}) = n \mathcal{I}(\tilde{\theta})e_j = 0$, and
    this holds at any $\tilde{\theta}$ with $\tilde{\theta}_j = \theta_j = 0$ by
    the assumption that $\{e_j, j: \lambda_j = 0\}$ spans the null space of
    $\mathcal{I}(\theta)$.

    Condition (iii) of Assumption \ref{assmp:pd} holds because we have assumed
    $k_j \leq 2$, which completes the proof.
  \end{proof}

  In the following two sections, we verify the conditions of Theorems
  \ref{thm:cont} and \ref{thm:main} in two mixed models. The first is an
  exponential mixed model with independent and identically distributed
  observations of a vector of correlated, positive responses. It has one scale
  parameter, one fixed effect parameter, and a non-normal random effect. This
  example illustrates our theory in non-linear mixed models with non-standard
  random effect distributions. The second model is a quite general version of
  the linear mixed model with normally distributed random effects. It has
  general design matrices $X_i$ and $Z_i$, possibly different for different $i =
  1, \dots n$, and hence non-identically distributed observations; and several
  fixed and random effect parameters.

\subsection{Exponential mixed model with uniform random effect}
Suppose the $Y_i$, $i = 1, \dots, n$, are independent copies of a $Y \in \R{2}$
which has conditionally independent elements given $W \in \R{}$ with conditional
densities
\begin{equation}\label{eq:exp}
  f_\theta(y_j\mid w) = (\psi + \lambda w)\exp\{-y_j(\psi + \lambda w)\}\I(y_j
  \geq 0); ~~ \theta = (\lambda, \psi) \in \Theta \subseteq \R{2},
\end{equation}
where $y_j$ is the $j$th element of $y \in \R{2}$; we omit observation indexes
for the remainder of the section and work only with the generic $Y\in \R{2}$.
The density $f_\theta(y_j\mid w)$ is that of an exponential random variable with
mean $1 / (\psi + \lambda w)$. We have selected $r = 2$ responses to simplify
calculations, but $r \geq 2$ presents no fundamental difficulties. The
specification clearly requires $\psi + \lambda W$ be positive almost surely.
There are many potentially useful specifications satisfying this, but to be
concrete suppose $W$ is uniform on $(-\sqrt{3}, \sqrt{3})$, so that $\nu(\dd w)
= f(w)\, \dd w$ with $f(w) = 12^{-1/2}\I(\vert w\vert \leq \sqrt{3})$; and the
parameter set is $\Theta = \{(\lambda, \psi) \in [0, \infty) \times \R{}: \psi >
\sqrt{3}\lambda \}$. The log-likelihood for one observation is, ignoring
additive constants,
\begin{align*}
  \log f_\theta(y) = \log \int f_\theta(y\mid w) \,\nu(\dd w) = \log \int (\psi + \lambda
  w)^2\exp\{-(\psi + \lambda w)y_{\bullet}\} \, \nu(\dd w),
\end{align*}
where $y_\bullet = y_1 + y_2$. The score for $\lambda$ is
\[
  s_\lambda(\theta; y) = -\frac{1}{f_\theta(y)}\int f_\theta(y \mid
  w)\{y_\bullet - 2 / (\psi + \lambda w)\}w\,\nu(\dd w).
\]
\begin{lem} \label{lem:exp_score_crit}
  In the exponential mixed model \eqref{eq:exp}, it holds for any $\theta = (0, \psi) \in \Theta$ that
  $\var_\theta\{s_\lambda(\theta, Y)\} = 0$ and $\var_\theta\{\nabla_1^2 \ell(\theta; Y)\} > 0$.
\end{lem}
\begin{proof}
  Setting $\theta = (0, \psi)$ in the expression for $s_\lambda(\theta; y)$ and
  moving terms that do not depend on $w$ outside the integral gives
  $s_\lambda(0, \psi; y) = -f_\theta(y)^{-1}\psi \exp(-\psi y_\bullet)\{y_\bullet
  - 2 / \psi\}  \int w \, \nu(\dd w) = 0$. Moreover, when $s_\lambda(\theta; y) = 0$,
  \[
    \nabla_1^2 \log f_\theta(Y)= -
    \frac{1}{f_\theta(Y)}\int f_\theta(Y\mid
    w)[\{Y_\bullet - 2/(\psi + \lambda w)\}^2 - 2 / (\psi +\lambda w)^2]w^2 \,\nu(\dd w),
  \]
  where $\log f_\theta(Y)$ means $f_\theta$ evaluated at $Y \sim f_\theta$. When
  $\lambda = 0$ this simplifies to $(Y_\bullet - 2/\psi)^2 - 2/ \psi^2$, which
  has positive variance under $\theta = (0, \psi)$.
\end{proof}

Lemma \ref{lem:exp_score_crit} shows $\theta = (0, \psi)$ is a critical point
for any $\psi$, agreeing with Corollary \ref{corol:indep_crit}. It also shows
the corresponding second derivative of the log-likelihood has positive variance
if $\lambda = 0$, suggesting it may be possible to verify Assumption
\ref{assmp:pd} with $k_1 = 2$ at $\theta$ with $\lambda = 0$ and $k_1 = 1$
elsewhere. The following lemma will be helpful to that end.

\begin{lem} \label{lem:exp_finf}
  In the exponential mixed model \eqref{eq:exp}, the information matrix
  $\mathcal{I}(\theta)$ has rank one if $\lambda = 0$ and rank two otherwise.
\end{lem}

\begin{proof}
  The score for $\psi$ is $s_\psi(\theta; y) = -{f_\theta(y)}^{-1}\int
  f_\theta(y \mid w)\{y_\bullet - 2 / (\psi + \lambda w)\}\, \nu(\dd w)$. Thus, when
  $\lambda = 0$, $s_\psi(\theta; Y) = Y_\bullet - 2/\psi$ which has positive
  variance. In conjunction with Lemma \ref{lem:exp_score_crit}, this shows the
  rank of $\mathcal{I}(\theta)$ is one when $\lambda = 0$. Suppose $\lambda > 0$
  and make the change of variables $t = (\psi + \lambda w)y_\bullet$ in the
  integral in the definition of the log-likelihood. Letting $G(t) = e^{-t}(t^2 +
  2t + 2)$, which is an antiderivative of $g(t) = -t^2e^{-t}$, gives \[ \log
  f_\theta(y) = -\log(\lambda)+ \log \{G(\psi - \sqrt{3}y_\bullet \lambda) -
  G(\psi + \sqrt{3}y_\bullet \lambda)\}. \] Differentiating with respect to
  $\lambda$ and $\psi$, taking an arbitrary linear combination given by $v =
  [v_1, v_2]^\tsp \in \R{2}$, and observing the (Lebesgue) density for
  $Y_\bullet$ is positive on $(0, \infty)$ shows $\mathcal{I}(\theta)v = 0$ only
  if
  \[
    t\mapsto \frac{v_1\lambda^{-1} t\{g(\psi - t) + g(\psi + t)\} - v_2\{g(\psi
    - t) - g(\psi + t)\}}{G(\psi - t) - G(\psi + t)}
  \]
  is constant on $(0, \infty)$, possibly except on a Lebesgue null set.
  Verifying this map is indeed non-constant on a set of positive Lebesgue
  measure is routine so we omit the details.
\end{proof}

We are ready to verify Assumptions \ref{assmp:support}--\ref{assmp:dense},
giving the main result of the section.

\begin{thm}
  The conclusions of Theorems \ref{thm:cont} and \ref{thm:main}
  hold in the exponential mixed model \eqref{eq:exp}.
\end{thm}

\begin{proof}
  Assumption \ref{assmp:support} holds with $\gamma$ being Lebesgue measure
  since $f_\theta(y) > 0$ for all $y \in (0, \infty)^2$ and $\theta \in \Theta$.

  To verify the moment condition in Assumption \ref{assmp:deriv}, it suffices by
  Lemma \ref{lem:exp_finf} to find locally uniform bounds of $\int \vert
  s_\lambda(\tilde{\theta}; y)\vert^3 f_\theta(y)\,\dd y$ and $\int \vert
  s_\psi(\tilde{\theta}; y)\vert^3 f_\theta(y)\,\dd y$ around an arbitrary
  $\theta' \in \Theta$, and of $\int \vert \nabla_1^2 \log f_{\tilde{\theta}}(y)
  \vert^{3} f_\theta(y)\, \dd y$ around $\theta'$ with $\lambda' = 0$. For the
  former, observe that for any $k \geq 1$, by Jensen's inequality and using
  $f_\theta(y\mid w) f(w) / f_\theta(y) = f_\theta(w\mid y)$
  \begin{align*}
    \int \vert s_\lambda(\tilde{\theta}; y)\vert^k f_\theta(y) \dd y &=
    \int \left\vert \int \{y_\bullet
    - 2 / (\tilde{\psi} + \tilde{\lambda} w)\}w \frac{f_{\tilde{\theta}}(y \mid w) f(w)}{f_{\tilde{\theta}}(y)}\,
    \dd w \right\vert^k f_\theta(y)\,\dd y \\
    &\leq  \int \int \left\vert
    \{y_\bullet - 2 / (\tilde{\psi} + \tilde{\lambda} w)\}w \right\vert^k
    f_{\tilde{\theta}}(w\mid y)\, \dd w\, f_\theta(y)\, \dd y.
  \end{align*}
  Now for any $\theta' \in \Theta$, since $\psi > \lambda \sqrt{3} \geq \lambda
  \vert w\vert$, we can find an $M < \infty$ and a ball $B$ centered at $\theta'$ small
  enough that $1 / (\psi + \lambda w) \leq M$ for all $(\psi, \lambda) \in B$
  and $w \in (-\sqrt{3}, \sqrt{3})$. Then for $\tilde{\theta} \in B$, by
  Jensen's inequality (twice), the inner integral less than $3^{k/2}2^{k -
  1}(\vert y_\bullet \vert^k + \vert 2 M\vert^k) \leq 3^{k/2}2^{k - 1}\{2^{k -
  1}(\vert y_1\vert^k + \vert y_2\vert^k) + \vert 2 M\vert^k)\}$. Thus, the last
  line in the last display is bounded uniformly for $(\tilde{\theta}, \theta)
  \in B \times B$ if $\E_\theta(Y_j^k)$ is, $j = 1, 2$. But for $ \theta \in
  B$, $\E_\theta(Y_j^k) = \E_\theta\{ \E_\theta(Y_j^k \mid W)\} = k!\E\{1 /
  (\psi + \lambda W)^k\} \leq k! M^k$. The other moment bounds can be handled
  similarly to show Assumption \ref{assmp:deriv} holds; we omit the details.

  Assumption \ref{assmp:loew_order} holds because the $Y_i$ are identically
  distributed.

  We use Proposition \ref{prop:assmp4} to verify Assumption \ref{assmp:pd}.
  Condition (i) of that proposition is by assumption and condition (ii) is
  established in Lemma \ref{lem:exp_finf}. To verify condition (iii), it
  suffices by Lemma \ref{lem:exp_finf} to show that, when $\lambda = 0$, $v_1
  \nabla_1^2 \log f_\theta(y) + v_2 s_\psi(\theta; y)$ is constant for almost
  every $y$ only if $v_1 = v_2 = 0$; this clearly holds since the linear
  combination is $v_1(y_\bullet - 2/\psi)^2 - 2 v_1/\psi^2 + v_2 (y_\bullet -
  2\psi)$.

  Assumption \ref{assmp:dense} holds because the only critical points are those
  with $\lambda = 0$, and this completes the proof.
\end{proof}

\subsection{Linear mixed models} \label{sec:lmm}

Consider a linear mixed model which assumes, for some $\psi \in \R{d_2}$ and
$\Lambda$ satisfying \eqref{eq:indep_scale}, independently for $i = 1, \dots, n$,
\begin{equation}\label{eq:lmm}
  Y_i \mid X_i, W_i \sim \rN(X_i \psi + Z_i \Lambda W_i, \sigma^2 I_{r_i}),
  \quad W_i\mid X_i \sim \rN(0, I_{r_i}),
\end{equation}
where $Z_i \in \R{r_i\times q}$ is a design matrix and $X_i \in \R{r_i\times
d_2}$ a matrix of predictors. The predictors can be non-stochastic or, more
generally, have a distribution not depending on $\theta$, possibly different for
different $i$. Assume also for simplicity that $\sigma^2 > 0$ is known. When
$\sigma^2 = 1$ and observations are identically distributed, this model is a
special case of the generalized linear mixed model in Example \ref{ex:glmm}.

We will first obtain a reliable confidence region for $\theta = (\lambda, \psi)
\in \Theta = [0, \infty)^{d_1} \times \R{d_2}$ by verifying the conditions of
Theorem \ref{thm:main}. Then we show how that result can be modified to give a
reliable confidence region for $\lambda$ only, with $\psi$ a nuisance parameter.
The model implies the distribution of $Y_i \mid X_i$ is multivariate normal with
mean $X_i \psi$ and covariance matrix
\[
  \Sigma_i = \Sigma_i(\lambda) = \sigma^2 I_{r_i} + \sum_{j = 1}^{d_1} \lambda_j^2
  H^i_{j},
\]
where $H^i_j = \sum_{k \in [j]} Z_i^k (Z_i^k)^\tsp \in \R{r_i\times r_i}$ is the
sum of outer products of columns of $Z_i$ corresponding to random effects scaled
by $\lambda_j$, $j = 1, \dots, d_1$. The log-likelihood is, ignoring additive
terms not depending on $\theta$,
\[
  \ell_n(\theta; y^n, X^n) = \sum_{i = 1}^n \left\{-\frac{1}{2}\log
  \vert\Sigma_i(\lambda) \vert - \frac{1}{2}(y_i - X_i \psi)^\tsp
  \Sigma_i(\lambda)^{-1}(y_i - X_i\psi)\right\},
\]
where $X^n = (X_1, \dots, X_n)$.
Differentiating with respect to $\lambda_j$ gives, for $j = 1, \dots, d_1$,
\begin{align*}
  s^i_j(\theta; y_i, X_i) =
  \lambda_j \tr\left\{\Sigma_i^{-1} H^i_j - \Sigma_i^{-1}(y_i - X_i \psi)(y_i - X_i
  \psi)^\tsp \Sigma_i^{-1} H^i_j \right\},
\end{align*}
which is equal to zero for every $i$ when $\lambda_j = 0$. Thus, scale
parameters at zero are critical points, agreeing with Corollary
\ref{corol:indep_crit}. We also have the following partial converse, essentially
saying that only scale parameters can be critical elements.

\begin{lem}\label{lem:blockdiag}
  The Fisher information $\mathcal{I}^i(\theta)$ for one observation in the
  linear mixed model \eqref{eq:lmm} is block-diagonal and, if
  $\emin\{\E(X_i^\tsp X_i)\} > 0$, is singular if and only if its leading block
  $\mathcal{I}^i_\lambda(\theta) = \cov_\theta\{ s^i_\lambda(\theta; Y_i,
  X_i)\}$ is, where $\emin(\cdot)$ is the smallest eigenvalue.
\end{lem}
The lemma holds as stated with non-stochastic predictors, but notation can be
simplified since $\E(X_i^\tsp X_i) = X_i^\tsp X_i$ in that case. The following
lemma gives a lower bound on the (conditional) variance of any linear
combination of the score for the scale parameters. Together with the previous
lemma, it can be used to identify all critical points.

\begin{lem} \label{lem:lmm_varbound}
 In the linear mixed model \eqref{eq:lmm}, for any $v \in \R{d_1}$,
 \[
  \var_\theta\{v^\tsp s^i_\lambda(\theta; Y_i, X_i)\mid X_i\} \geq 2
  \emax(\Sigma_i)^{-2} \emin(Z_i^\tsp Z_i)^2 \max_{j} (\lambda_j v_j)^2
 \]
 and
 \[
 \var_\theta\{v^\tsp s^i_\lambda(\theta; Y_i, X_i)\mid X_i\} \leq 2 r_i \emin(\Sigma_i)^{-2}
 \emax(Z_i^\tsp Z_i)^2 \max_{j} (\lambda_j v_j)^2,
 \]
 where $\emax(\cdot)$ is the largest
 eigenvalue.
\end{lem}
Together with Lemma \ref{lem:blockdiag}, Lemma \ref{lem:lmm_varbound} ensures
that, as long as $Z_i^\tsp Z_i$ and $\E(X_i^\tsp X_i)$ are invertible, $\{e_j,
j: \lambda_j = 0\}$ spans the null space of $\mathcal{I}^i(\theta)$. As observed
following Proposition \ref{prop:assmp4}, this makes checking the assumptions in
Section \ref{sec:general} easier. Similarly, an implication is
that, when implementing the proposed method, we can take $k_j(\theta) = 2$ when
$\lambda_j = 0$ and $k_j(\theta) = 1$ otherwise. We are ready for the first
main result of the section.

Let $\Vert \cdot \Vert$ be the spectral norm when the argument is a matrix.

\begin{thm}\label{thm:lmm}
  If $\limsup_{i\to\infty}r_i < \infty$ and there exists an $M \in (0, \infty)$
  such that, for every $i = 1, 2 \dots$, $ M^{-1} \leq \emin(Z_i^\tsp Z_i) \leq
  \emax(Z_i^\tsp Z_i) \leq M$, $M^{-1} \leq \emin\{\E(X_i^\tsp X_i)\}$, and
   $\E(\Vert X_i\Vert^{4 + \delta}) \leq M$ for some $\delta > 0$; then
  Assumptions \ref{assmp:support}--\ref{assmp:dense}, and hence the
  conclusions of Theorems \ref{thm:cont} and \ref{thm:main}, hold in the linear
  mixed model \eqref{eq:lmm}.
\end{thm}

\begin{proof}[Proof of Theorem \ref{thm:lmm}]
  We prove the theorem with non-stochastic $X_i$; the calculations for
  stochastic $X_i$ are similar. Then, since $\E(\Vert X_i\Vert^{4 + \delta}) =
  \Vert X_i\Vert^{4 + \delta}$ we may assume $\emax(X_i^\tsp X_i)\leq M$, upon
  increasing $M$ as needed. Assumption \ref{assmp:support} is satisfied since
  $f_\theta(y_i)$ is strictly positive on $\R{r_i}$ for every $\theta \in
  \Theta$, with $\gamma_i$ being Lebesgue measure. We verify Assumption
  \ref{assmp:deriv} with $\delta > 0$ and $k = 2$ at critical $\theta'$. Define
  $\xi^i(\theta; y_i, X_i) \in \R{d}$ by
  \[
    \xi^i_j(\theta; y_i, X_i) = \tr\left\{\Sigma_i^{-1}H^i_j - \Sigma_i^{-1}(y_i
    - X_i\psi)(y_i - X_i\psi)^\tsp\Sigma_i^{-1}H^i_j\right\},~ j \leq d_1,
  \]
  and the remaining $d_2$ elements equal to $s_\psi^i(\theta; y_i,X_i)$. When
  $\lambda_j > 0$, $\xi^i_j(\theta; y_i, X_i) = s^i_j(\theta; Y_i, X_i) /
  \lambda_j$. Because $\xi^i_j(\theta; y_i, X)$ is continuous in $\theta$
  since $\sigma^2 > 0$, when $\lambda_j = 0$ it holds that $\xi_j^i(\theta; y_i,
  X_i) = \nabla_j^2 \ell^i(\theta; y_i, X_i)$. More generally,
  for $j, l \leq d_1$,
  \begin{align*}
     \nabla^2_{jl} \ell^i(\theta; y_i, X_i) &= \I(j =
    l) \xi_j^i(\theta; y_i, X_i) - \lambda_j \tr\left\{\Sigma_i^{-1}
    H^i_k \Sigma_i^{-1} H^i_j\right\} \\ &~~ + \lambda_j \tr\left\{\Sigma_i^{-1} H^i_k
    \Sigma_i^{-1}(y_i - X_i \psi)(y_i - X_i \psi)^\tsp \Sigma_i^{-1} H^i_j\right\}\\ &~~ +
    \lambda_j\tr\left\{\Sigma^{-1}(y_i - X_i \psi)(y_i - X_i \psi)^\tsp \Sigma_i^{-1}
    H^i_k \Sigma_i^{-1} H_j^i\right\}.
  \end{align*}
  Since $\sigma^2 > 0$,we can pick for any $\theta' \in
  \Theta$ a small enough open ball $B \subseteq \R{d}$ centered
  at $\theta'$ and large enough $M < \infty$ to have, on $B$, that $\sigma^{-2} \leq M$, $\Vert \Sigma_i^{-1}\Vert \leq M$, and $\lambda_j \leq M$
  for all $j$. Thus, using sub-multiplicativity of the spectral norm and that
  the trace is upper bounded by $r_i$ times the spectral norm, on $B$,
  \begin{align*}
    |\nabla^2_{jl}\ell^i(\theta; y_i, X_i)| &\leq r_i \{ M \Vert H_j^i\Vert +
    M^2 \Vert H_j^i\Vert \Vert y_i - X_i\psi\Vert^2 \} + r_i M^3 \Vert
    H_l^i\Vert \Vert H_{j}^i\Vert \\ &~~ + 2 r_i M^4\Vert y_i - X_i\psi\Vert^2
    \Vert H_j^i\Vert \Vert H_l^i\Vert.
  \end{align*}
  Now, $\Vert H_j^i\Vert \leq \Vert Z_i^\tsp Z_i\Vert \leq M$ and $\Vert y_i -
  X_i \psi\Vert \leq \Vert y_i\Vert + \Vert X_i \psi\Vert\leq \Vert y_i\Vert + M
  \Vert \psi\Vert$ for all $(i, j)$. Thus, since $\Vert \psi\Vert$ is bounded on
  $B$, the expectation in Assumption \ref{assmp:deriv} to be bounded is, upon
  increasing $M$ as needed, less than
  \[
    M \{1 + \sup_{\theta \in \Theta\cap B}\E_\theta(\Vert Y_i\Vert^{4 + 2 \delta}),
  \]
  which is bounded uniformly in $i$ since $Y_i$ is a normal vector whose mean
  and covariance matrix are bounded uniformly in $i$ and $\theta \in B\cap
  \Theta$. The calculations for $k = 1$ and $j> d_1$ or  $l > d_1$ are very
  similar and hence omitted.

  To verify Assumption \ref{assmp:loew_order}, it suffices to find $c_1, c_2$
  such that, for any $v \in \R{d}$,
  \[
    c_1(\theta)v^\tsp \mathcal{I}^{1}(\theta) v \leq v^\tsp \mathcal{I}^i(\theta) v \leq c_2(\theta)v^\tsp \mathcal{I}^{1}(\theta)v.
  \]
  Let $v = [v_1^\tsp , v_2^\tsp]^\tsp$ with $v_1 \in \R{d_1}$. By Lemmas
  \ref{lem:blockdiag} and \ref{lem:lmm_varbound}, using that $r_i$ and the
  eigenvalues of $X_i^\tsp X_i$, $Z_i^\tsp Z_i$, and $\Sigma_i^{-1}$ are
  bounded, there exist $M \in (0, \infty)$ and $c_3, c_4 : \Theta \to (0,
  \infty)$ not depending on $i$ such that
  \[
    M^{-1}\{\Vert v_2\Vert + c_3(\theta) \max_{j \leq d_1} (\lambda_j v_j)^2\}
    \leq v^\tsp \mathcal{I}^i(\theta) v \leq M \{\Vert v_2\Vert +
    c_4(\theta) \max_{j \leq d_1 }(\lambda_j v_{j})^2\},
  \]
  Thus, the desired inequalities hold with $c_1(\theta) = M^{-2}\{1 +
  c_4(\theta)\}^{-1}\min\{1,c_3(\theta)\}$ and $c_2(\theta) = M^2 \min\{1,
  c_3(\theta)\}^{-1}\{1 + c_4(\theta)\}$.

  To verify Assumption \ref{assmp:pd}, we argue as in the proof of Proposition
  \ref{prop:assmp4}, with minor modifications to take the observation index into
  account.  Thus, adding an observation
  index to the vector $\bar{s}$ defined in Proposition \ref{prop:assmp4}, and
  making the dependence on the predictors explicit, we get that $\xi^i(\theta;
  Y_i, X_i)$ is equal to $\bar{s}^i(\theta; Y_i, X_i)$ with each element
  corresponding to a $\lambda_j > 0$ scaled by $1 /  \lambda_j$. Thus, the
  covariance matrix of $\xi^i(\theta; Y_i, X_i)$ is positive definite if and
  only if that of $\bar{s}^i(\theta; Y_i, X_i)$ is. With that, the verification
  of Assumption \ref{assmp:pd} follows as in the proof of Proposition
  \ref{prop:assmp4}.

  Assumption \ref{assmp:dense} holds because, for $j = 1, \dots d_1$,
  $\lambda_j = 0$ is the only critical element in $[0, \infty)$, and that
  completes the proof.
\end{proof}

Lastly in this section we consider confidence regions for $\lambda$ only. To
that end, let
\[
  T_n^\lambda(\lambda; \psi, Y^n, X^n) = \left\{\sum_{i = 1}^n
  s^i_\lambda(\theta; Y_i, X_i)^\tsp \right\}\mathcal{I}_n^\lambda(\theta)^{-1}
  \left\{ \sum_{i = 1}^n s^i_\lambda(\theta; Y_i, X_i) \right\},
\]
where $\mathcal{I}_n^\lambda(\theta) = \cov_\theta\{s_n^\lambda(\theta; Y^n,
X^n)\}$ is the leading $d_1\times d_1$ block of $\mathcal{I}_n(\theta)$. For
such regions to be practically useful, or feasible, $\psi$ has to be known or
estimated. Our next result says $\psi$ can estimated by any square root
$n$-consistent estimator without affecting the asymptotic coverage probability;
this result assumes identically distributed observations for simplicity.
To be more specific, let
\[
  \mathcal{R}^{\lambda}_n(\alpha) = \{\lambda : T_n^\lambda(\lambda; \psi, Y^n,
  X^n) \leq q_{d_1, 1 - \alpha}\}
\]
and let $\hat{\mathcal{R}}^{\lambda}_n(\alpha)$ be that confidence region with
an estimator $\hat{\psi}$ in place of $\psi$ in $T_n$.

\begin{thm}\label{thm:est:cover}
  If the $(Y_i, X_i)$, $i = 1, 2, \dots$, are identically distributed and the
  conditions of Theorem \ref{thm:lmm} hold, the confidence region
  $\mathcal{R}^{\lambda}_n(\alpha)$ satisfies, for any compact $C \subseteq
  \Theta$ and $\alpha \in (0, 1)$,
  \[
    \lim_{n\to \infty}\inf_{\theta \in C}\pr_\theta\left\{\lambda \in \mathcal{R}^{\lambda}_n(\alpha) \right\} = 1- \alpha;
  \]
  and if in addition $\sqrt{n}\Vert \hat{\psi}_n - \psi_n \Vert = O_{\pr}(1)$
  under any convergent $\{\theta_n = (\lambda_n, \psi_n)\} \in \Theta$, then the same holds for
  $\hat{\mathcal{R}}^\lambda_n(\alpha)$.

\end{thm}

In practice, the estimator $\hat{\psi}_n$ can be, for example, the least squares
estimator
\[
  \hat{\psi}_n = \left(\sum_{i = 1}^n X_i^\tsp X_i\right)^{-1}\sum_{i = 1}^n
  X_i^\tsp Y_i.
\]
This estimator is square root-$n$ consistent under convergent $\{\theta_n\}$
under weak conditions. More specifically, it is multivariate normally
distributed given $(X_1, \dots, X_n)$ with mean $\psi_n$ and covariance matrix
\[
  \left(\sum_{i = 1}^n X_i^\tsp X_i\right)^{-1}\left(\sum_{i = 1}^n X_i^\tsp
  \Sigma_n X_i\right)\left(\sum_{i = 1}^n X_i^\tsp X_i\right)^{-1},
\]
which tends to a matrix of zeros under weak conditions.

\subsection{Practical considerations}

In many mixed models the likelihood does not admit an analytical expression,
which complicates implementing likelihood-based methods in
practice. In particular, evaluating the log-likelihood, or its derivatives and
their expectations, may require the numerical evaluation or
approximation of integrals. A detailed treatment of numerical integration is
outside the scope of the present article, but we briefly discuss two possible
practical issues: (i) the critical points and their corresponding critical
vectors are known, but the proposed test statistic does not admit a convenient
expression at those points; and (ii) it is difficult to establish which points
are critical by analytical means. For clarity, we discuss identically
distributed observations and assume $k_j \leq 2$ suffices to satisfy Assumption
\ref{assmp:pd}, which in our experience is often the case in practice; other
settings can be handled similarly.

Issue (i) can occur when, for example, Corollary \ref{corol:indep_crit} says a
scale parameter $\lambda_j = 0$ is a critical point with critical vector $e_j$
in a generalized linear mixed model. To compute our test statistic at such
points it is useful to note that, with the $\bar{s}$ defined in Proposition
\ref{prop:assmp4} and under regularity conditions,
\[
  T_n(\theta; Y^n) = n^{-1}\left\{\sum_{i = 1}^n \bar{s}(\theta; Y_i)\right\} \cov_\theta\{ \bar{s}(\theta; Y_1)\}^{-1} \left\{\sum_{i = 1}^n \bar{s}(\theta; Y_i)\right\}.
\]
Thus, the user needs to compute first and second order partial derivatives of
the log-likelihood and their covariance matrix. A routine calculation shows,
under regularity conditions, the $j$th element of
$\bar{s}(\theta, y_i)$ is
\begin{equation}\label{eq:comp_deriv}
\frac{1}{f_\theta(y_i)}\int \nabla_{j}^{k_j}
  f_{\theta}(y_i\mid w) \nu(\dd w),
\end{equation}
where $k_j = 2$ for $\theta_j = \lambda_j = 0$ and $k_j = 1$ otherwise. The
derivative inside the integral can typically be computed analytically for any
$k_j$ and hence there are no fundamental computational differences between $k_j
= 1$ and $k_j = 2$. We illustrate these calculations in detail
in an example in the Supplementary Material. Software
packages for mixed models typically approximate integrals like that in
\eqref{eq:comp_deriv} with $k_j = 1$ using Laplace approximations, adaptive
Gaussian quadrature, or Monte Carlo
\citep[e.g][]{Bates.etal2015,Knudson.etal2021}. The integrals with $k_j = 2$ and
those required to compute the covariance matrix can be handled by similar
methods. When the critical vectors are known but are not standard basis vectors,
the computations are similar but with different linear combinations of the
derivatives.

By contrast, issue (ii) is not merely computational. If one cannot establish the
existence or lack of critical points, then one does not know whether the Fisher
information is invertible and, hence, whether the theory motivating many other
methods applies. This highlights a practical advantage of the proposed
procedure: if the inversion of the Fisher information fails, the user is
effectively warned they are attempting inference at a previously unknown
critical point. Having identified a critical point, the critical vectors can be
calculated numerically by spectral decomposition of the Fisher information, and
then one is essentially back to issue (i).

\section{Numerical experiments} \label{sec:sims}

We examine finite sample properties of the proposed method in a linear mixed
model. The Supplementary Material
contains similar simulations in a generalized linear mixed model for binary
responses with asymmetrically distributed random effects, and the results there
are qualitatively similar.

Suppose that for $i
= 1, \dots n$ and $j = 1, \dots, r$,
\begin{equation}\label{eq:sim}
  Y_{ij} = \psi_1 + \psi_2 X_{ij} + U_{1i} + U_{2i}
X_{ij} + E_{ij},
\end{equation}
where $[U_{1i}, U_{2i}]^\tsp \sim \rN\{0, \diag(\lambda_1^2, \lambda_2^2)\}$,
independently for $i = 1, \dots, n$ and independent of $E = [E_{1, 1}, \dots,
E_{n, r}]^\tsp \sim \rN(0, \sigma^2 I_{nr})$. To conform with the previous
section, we assume $\sigma^2 = 1$ is known. Thus, the parameter set is $\Theta =
\{\theta = (\lambda, \psi) \in[0, \infty)^2 \times \R{2}\}$. We study confidence
regions for $\lambda$ with $\psi$ estimated, that is, a nuisance parameter. The
Supplementary Material contains simulations where
$\psi$ is known, and simulations where $\sigma$ is unknown and confidence
regions are created for $(\lambda, \sigma)$ with $\psi$ estimated. We compare
the proposed confidence region to those obtained by inverting the likelihood
ratio test statistic and the Wald test statistic standardized by expected
information evaluated at the (unconstrained) maximum likelihood estimates.
Specifically,
\[
  T_n^L(\lambda; y^n) = 2 \{\ell_n(\hat{\theta}, y^n) - \ell_n([\lambda,
  \tilde{\psi}]^\tsp, y^n)\},
\]
where $\hat{\theta} \in \argmax_{\theta \in \Theta}\ell_n(\theta; y^n)$ and
$\tilde{\psi} = \tilde{\psi}(\lambda) \in \argmax_{\psi \in \R{2}}\ell_n([\lambda,
\psi]^\tsp, y^n)$; and
\[
  T_n^W(\lambda; y^n) = (\hat{\lambda} - \lambda)^\tsp
  \mathcal{I}^\lambda_n(\hat{\theta})(\hat{\lambda} - \lambda),
\]
where $\mathcal{I}_n^\lambda(\theta)$ is the $2 \times 2$ block of
$\mathcal{I}_n(\theta)$ corresponding to $\lambda$.

Our theory suggests the proposed confidence region may have good finite sample
coverage near critical points because the proposed test statistic has the same
asymptotic distribution under any sequence of parameters tending to a critical
point as the sample size increases. Conversely, we expect the coverage of the
other confidence regions may depend on how close to a critical point the true
parameter is. To examine this we consider true
$\lambda$ at different distances from the origin:
\[
  \lambda_1 = \lambda_2 \in \{10^{-6}, 0.01, 0.05, 0.1, 0.2, 0.3, 0.4, 0.5\}.
\]
Because these $\lambda$ correspond to interior points of the parameter set, a
level $1 - \alpha$ confidence region covers $\lambda$ if the test statistic at
$\lambda$ is smaller than the $(1 - \alpha)$th quantile of the chi-square
distribution with two degrees of freedom. The use of that distribution is
motivated by classical asymptotic theory for $T_n^L(\lambda; Y^n)$ and
$T_n^W(\lambda; Y^n)$, and by our theory for the proposed test statistic.
Different reference distributions should be used for the Wald and likelihood
ratio statistics at boundary points
\citep{Self.Liang1987,Geyer1994,Baey.etal2019}, while the proposed method uses
the same reference distribution at every point of the parameter set.

We performed a Monte Carlo experiment with 10,000 replications. We set $\psi =
1_2$ and, for every $(\lambda, \psi)$ considered, generated stochastic
predictors as independent draws from a uniform distribution on $[-1, 2]$.
Responses were then generated according to \eqref{eq:sim}. The sample sizes were
$n \in \{20, 80\}$ and $r = 10$. Code for reproducing the results is available
at \url{https://github.com/koekvall/conf-crit-suppl}.

Figure \ref{fig:sim_cover} summarizes the results. Notably, the proposed
confidence region has near-nominal estimated coverage probability in all
considered settings. By contrast, the estimated coverage probabilities for the
likelihood ratio and Wald confidence regions are substantially different from
nominal for many settings. Moreover, their coverage is sometimes lower and
sometimes higher than nominal. That is, for those methods, the quality of the
chi-square distribution as a reference distribution depends on how close to the
critical point the true parameter is. Simulations in the Supplementary Material
indicate the proposed method give near-nominal
coverage probabilities also of critical points, including ones at the boundary.

Figure \ref{fig:sim_qq} examines the agreement between sample quantiles of the
three considered test statistics and the theoretical quantiles of a chi-square
distribution with two degrees of freedom. The first plot, corresponding to small
but non-zero scale parameters, shows the agreement is poor near critical points
for the likelihood ratio and Wald test statistics. For the larger scale
parameters, agreement between sample and theoretical quantiles is decent for all
three test statistics.

\begin{figure}[h]
 \centering
 \includegraphics[width = 0.8\textwidth]{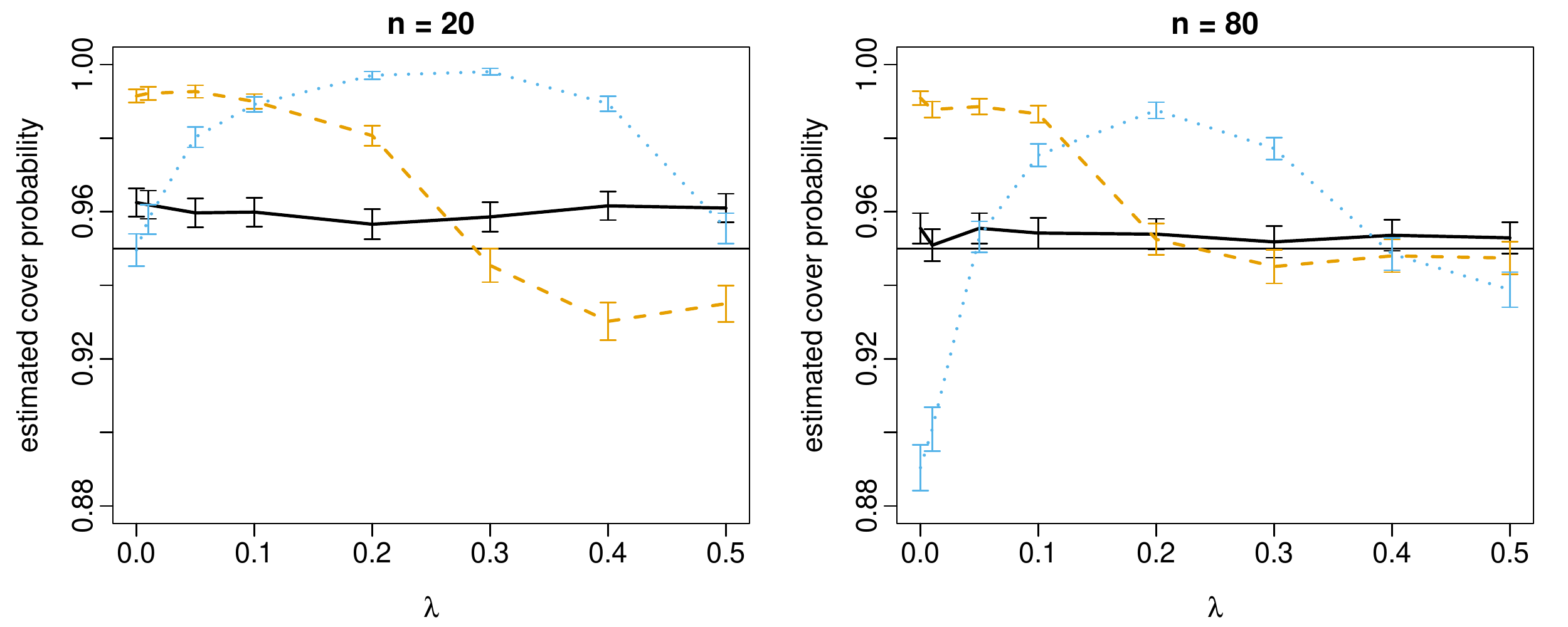}

 \caption{Monte Carlo estimates of coverage probabilities of confidence regions
 from inverting the modified score (solid), likelihood ratio (dashed), and Wald
 (dotted) test statistics. The straight horizontal line indicates the nominal
 0.95 coverage probability and vertical bars denote $\pm 2$ times Monte Carlo
 standard errors.}
 \label{fig:sim_cover}
\end{figure}

\begin{figure}[h]
 \centering
 \includegraphics[width = 0.8\textwidth]{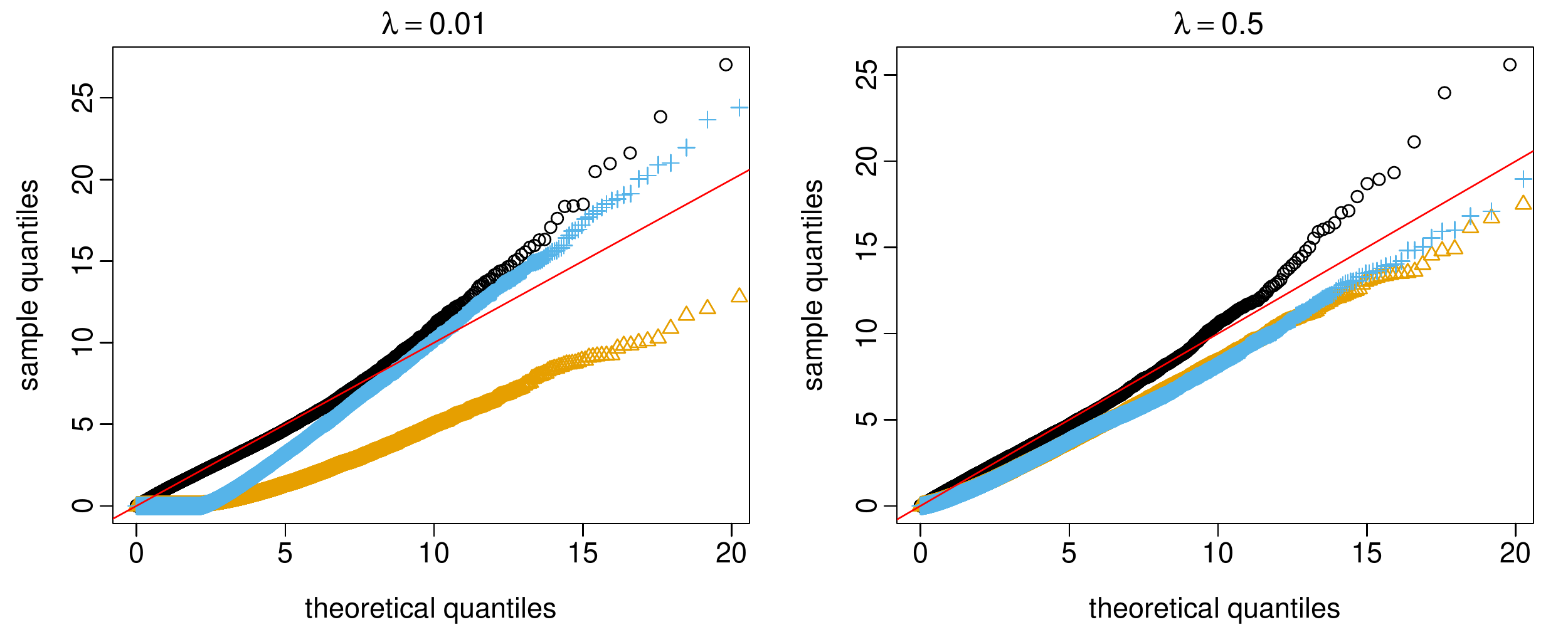}

 \caption{Quantile-quantile plots for modified score (circles), likelihood ratio
 (triangles), and Wald (plus signs) test statistics. The theoretical quantiles
 are from the chi-square distribution with 2 degrees of freedom and the sample
 quantiles from 10,000 Monte Carlo replications with $(n, r) = (80, 10)$.}
 \label{fig:sim_qq}
\end{figure}

For additional insight into test statistics' behavior near critical points,
Figure \ref{fig:sim_power} shows estimated rejection probabilities (size
and power) for tests of the null hypothesis that $\lambda_1 = \lambda_2 =
10^{-6}$. The data generating settings are the same as those used for Figure
\ref{fig:sim_cover}. The power curves are not directly comparable because, as
was also shown in Figure \ref{fig:sim_cover}, the different tests have different
sizes. Nevertheless, the power curves behave similarly as the true $\lambda$
moves away from the null hypothesis value. This indicates the differences in
coverage observed in Figure \ref{fig:sim_cover} is not in general due to how
large the different confidence regions are.

\begin{figure}[h]
 \centering
 \includegraphics[width = 0.8\textwidth]{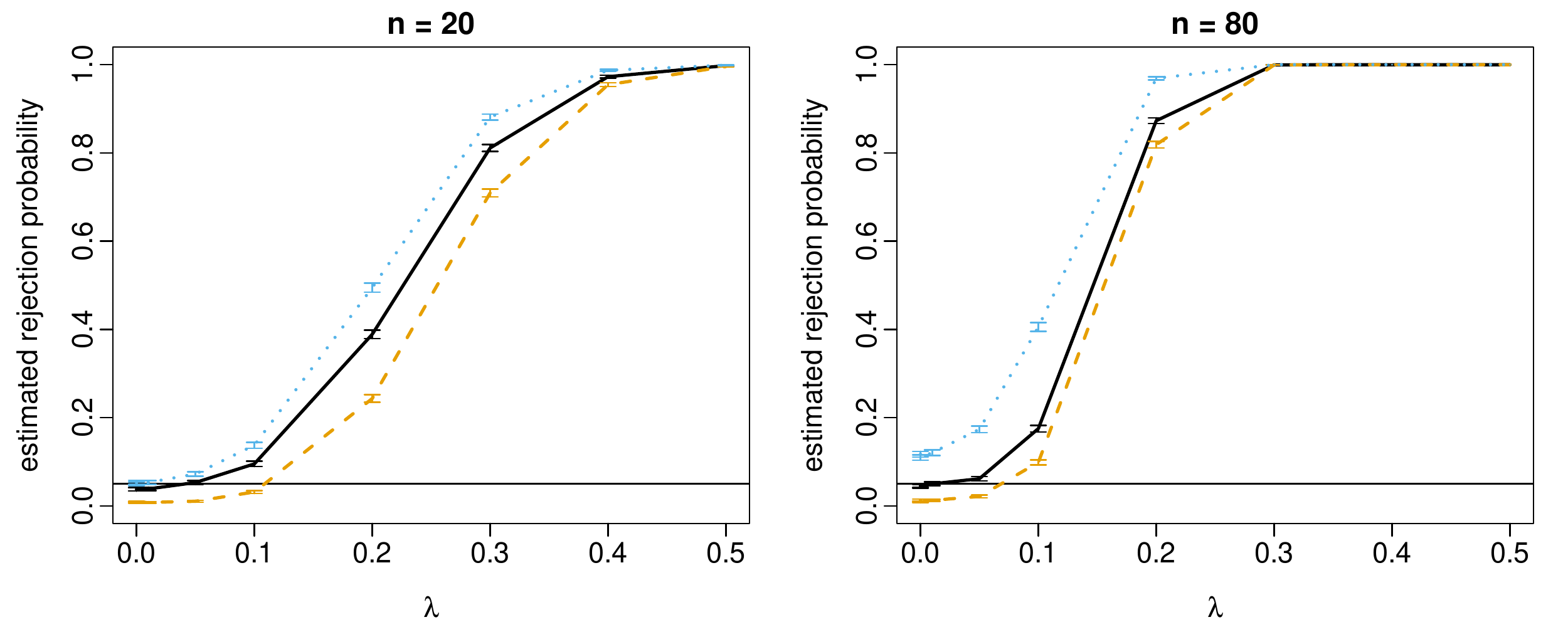}
 \caption{Monte Carlo estimates of rejection probabilities for testing the null
 hypothesis $\lambda = 10^{-6} 1_2$ using the modified
 score (solid), likelihood ratio (dashed), and Wald (dotted) test. The straight
 horizontal line indicates the size of the tests (0.05) and vertical bars denote
 $\pm 2$ times Monte Carlo standard errors. }
 \label{fig:sim_power}
\end{figure}

We considered several configurations in addition to those reported, including
both larger and smaller values of $n$ and $\lambda$, and the results were
remarkably consistent. To compute the proposed test statistic, we used the {\tt
lmmstest} R package written by the first author. To fit the model and compute
the likelihood ratio and Wald test statistics we used the {\tt lme4} R package
\citep{Bates.etal2015}. The Supplementary Material
includes times for computing the test-statistics in
the simulations on which Figure \ref{fig:sim_cover} is based. The proposed
test-statistic was about 3--4 times faster to evaluate than the likelihood ratio
and Wald statistics on average, but we expect both relative and absolute
computing times to vary substantially between settings and implementations.

We also note there is a Stata routine for calculation of the proposed
confidence region for a single variance parameter in linear mixed models with a
random intercept \citep{Bottai.Orsini2004}, a special case which can also be
treated using exact finite sample methods
\citep{Crainiceanu.Ruppert2004}.

\section{Data example} \label{sec:datex}

We illustrate using a dataset presented by \citet{Fitzmaurice.etal2012} which
contains a subset of the pulmonary function data collected in the Six Cities
Study \citep{Dockery.etal1983}. The data include a pulmonary measure called the
forced expiratory volume in the first second ({\tt FEV1}), height ({\tt ht}),
and age obtained from a randomly selected subset of the female participants
living in Topeka, Kansas. The sample includes $n = 300$ girls and young women,
with a minimum of one and a maximum of twelve observations over time. Age and
height are believed to be associated with the ability to take in and force out
air. To model these data, consider the linear mixed model
\begin{equation} \label{eq:data_ex:model}
\mathtt{FEV1}_{it} = \psi_1 + \psi_2 \mathtt{age}_{it} + \psi_3
  \mathtt{ht}_{it} + \psi_4 \mathtt{age}_{i1} + \psi_5
  \mathtt{ht}_{i1} + U_{i1} + U_{i2} \mathtt{age}_{it} + E_{it},
\end{equation}
where $U_{i1} \sim \rN(0, \lambda_1^2)$, $U_{i2} \sim \rN(0, \lambda_2^2)$, and
$E_{it} \sim \rN(0, \sigma^2)$ are mutually independent for all individuals
indexed by $i$ and time points indexed by $t$. This is a model considered by
\citet{Fitzmaurice.etal2012}, modified slightly to fit our setting. The
parameter set is $\Theta = \{\theta = (\psi, \lambda, \sigma) \in \R{5} \times
[0, \infty)^2 \times (0, \infty)\}$.

\begin{table}[h]
\centering
\begin{tabular}{l l l l l}
 Parameter & Estimate & Mod. Score CI & Lik. Rat. CI & Wald CI \\
 \hline
 $\lambda_1$ (intercept) & $0$ & $(0, 0.0482)$ & $(0, 0.0604)$ & $(0, 0.126)$\\
 $\lambda_2$ (age) & $0.0201$ & $(0.0188, 0.0219)$ & $(0.0183, 0.0222)$ & $(0.0181, 0.0222)$ \\
 $\sigma$ (error) & $0.156$ & $(0.152, 0.162)$ & $(0.151, 0.162)$ & $(0.151, 0.162)$\\
\end{tabular}

\caption{Maximum likelihood estimates and 95\% confidence intervals (CI) for
scale parameters in the linear mixed model \eqref{eq:data_ex:model}.} \label{tab:fit}
\end{table}

The maximum likelihood estimate of $\psi$ is $\hat{\psi} = (-2.2, 0.078, 2.80,
-0.040, -0.19)$, but we focus on the scale parameters whose estimates are in
Table \ref{tab:fit}. Notably, the maximum likelihood estimate of $\lambda_1$ is
zero, indicating common confidence regions may be unreliable. We present a
confidence region for $\lambda$ (Figure \ref{fig:data_ex}) and three
componentwise confidence regions (Table \ref{tab:fit}). When creating a
confidence region for a sub-vector or component of $\theta$, say $\lambda_1$,
the other components are effectively nuisance parameters. Then, at non-critical
points we standardize the score for $\lambda_1$ by the Schur complement
(efficient information) $\mathcal{I}_n / \mathcal{I}_n^{-\lambda_1}$, where
$\mathcal{I}_n^{-\lambda_1}$ is $\mathcal{I}_n$ with rows and columns
corresponding to $\lambda_1$ removed and evaluate at estimates of the nuisance
parameters (see e.g. \citet{Fewster.Jupp2013} or \citet[][Chapter
2]{Bickel.etal1998} for motivation). Similarly, at critical points we
standardize the modified score for $\lambda_1$ by $\tilde{\mathcal{I}}_n /
\tilde{\mathcal{I}}_n^{-\lambda_1}$; we comment further on nuisance parameters
in Section \ref{sec:final}.

For context, we also present likelihood ratio intervals based on the profile
likelihood computed using the {\tt confint} function in {\tt lme4}, and Wald
intervals based on maximum likelihood estimates from {\tt lme4} and the expected
Fisher information evaluated at estimates. These use the chi-square distribution
with one degree of freedom as reference. To decide whether to include the
boundary points $\lambda_1 = 0$ and $\lambda_2 = 0$ in the componentwise
confidence intervals, we note $\hat{\lambda}_1 = 0$ so $0$ is in confidence
intervals based on the profile likelihood or Wald statistic for any reference
distribution. Conversely, the test statistics for $\lambda_2$ are so large at
$0$ that $0$ should not be in either region for any relevant reference
distribution. To validate this, we used the {\tt varTestnlme} R package
\citep{Baey.Kuhn2019} to test, separately, $\lambda_1^2 = 0$ and $\lambda_2^2 =
0$, and got the $p$-values $0.5$ and $4\times 10^{-78}$, respectively.

The proposed interval for $\lambda_1$ is substantially smaller than that based
on the likelihood ratio (Table \ref{tab:fit}). This is consistent with our
simulations where the latter had greater than nominal empirical coverage of
small scale parameters. The Wald interval for $\lambda_1$ is even wider than the
likelihood ratio-based interval. The proposed interval for $\lambda_2$ is
smaller than the other two, and its left endpoint is further from zero. Thus,
there are indications the proposed procedure leads to not only reliable but more
precise inference. The intervals for $\sigma$ only differ in the third
significance digit. This is consistent with both theory and simulations since
the estimate of $\sigma$ is further from zero than those of $\lambda_1$ and
$\lambda_2$, and hence the different test statistics are expected to behave
similarly.

Figure \ref{fig:data_ex} shows plots of the componentwise test statistics based
on the proposed method for a range of $\lambda_1$ and $\lambda_2$. The values of
$\lambda_1$ and $\lambda_2$ such that the graph of the corresponding test
statistic is below the critical value $3.84$, the $0.95$th quantile of the
chi-square distribution with one degree of freedom, give the confidence regions
in Table \ref{tab:fit}. The graphs indicate the test statistics are convex in
$\lambda_1$ and $\lambda_2$, respectively. The proposed regions for $\lambda$
(Figure \ref{fig:data_ex}, third plot), which use the
chi-square distribution with two degrees of freedom as reference distribution,
can be used to assess which values of $\lambda$ are supported by the data. We
may, for example, reject the joint null hypothesis that $\lambda_1 = \lambda_2 =
0$ at conventional levels of significance. To create these
graphs and the corresponding confidence regions in Table \ref{tab:fit}, we
evaluated the test statistics at a grid of 50 values each for $\lambda_1$ and
$\lambda_2$ and included in the confidence regions those points where the
test-statistics were less than the desired quantile of the reference
distributions. Thus, for the first and second plot we evaluated componentwise
test-statistics 50 times each, and for the third plot we evaluated the
test-statistic for $\lambda$ at $50\times 50 = 2500$ points. The coarseness of
the grid can be adjusted depending on desired accuracy and computing times.

\begin{figure}[h]
 \centering
 \includegraphics[width = 0.9\textwidth]{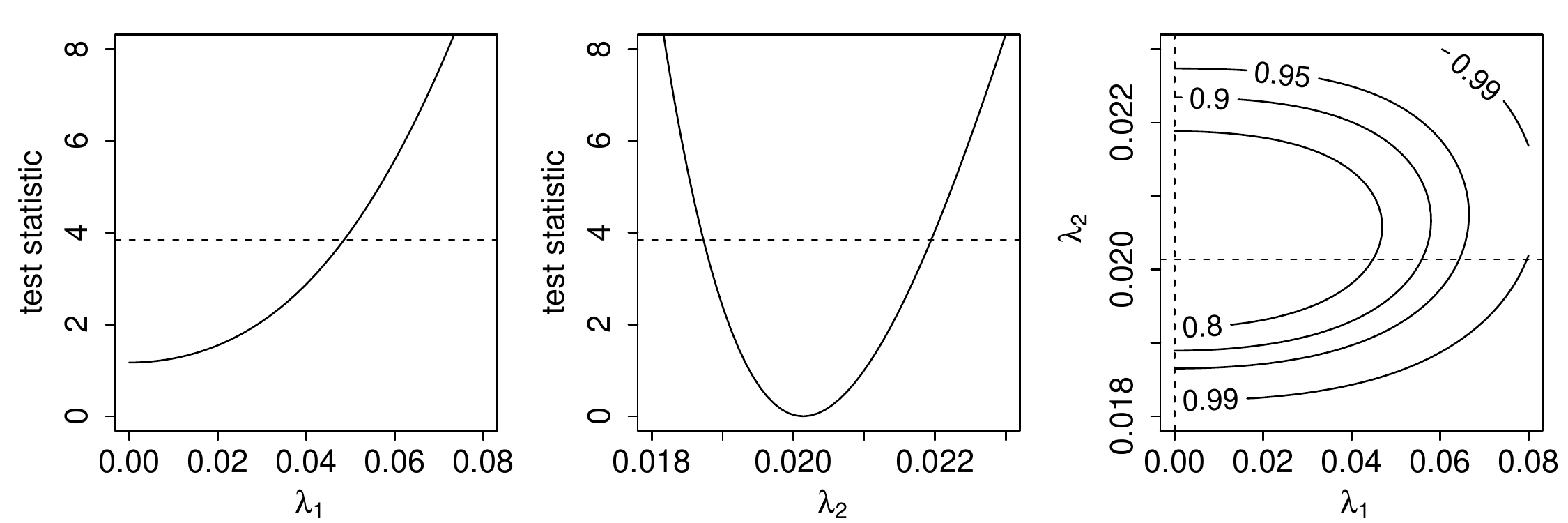}

 \caption{Componentwise test statistics for $\lambda_1$ (first plot) and
 $\lambda_2$ (second plot), and joint 80--99 \% confidence regions for $\lambda$
 (third plot), based on the chi-square distribution with two
 degrees of freedom. The dashed lines in the first two plots mark the 0.95th
 quantile of a chi-square distribution with one degree of freedom. The dashed
 lines in the third plot mark maximum likelihood estimates.}
 \label{fig:data_ex}
\end{figure}

\section{Final remarks} \label{sec:final}

Linking the boundary problem to the singular information problem allows a deeper
understanding of the behavior of the likelihood function in shrinking
neighborhoods of the boundary of a parameter set. Perhaps more importantly, it
permits the construction of confidence regions that have asymptotically correct
uniform coverage probability.

The advantages of using the proposed modified score test in constructing
confidence regions are many-fold: the proposed procedure does not
in general require a consistent point estimator,
which can be troublesome when the parameter is at or near the boundary; it does
not in general rely on simulation algorithms, which typically need to be
programmed for the specific problem at hand; it can be applied to a broad
variety of models, including the linear mixed model; it allows inference
on scale parameters when the random effects follow an asymmetric distribution,
which gives insight about the sign and the magnitude of the skewness. In
addition, to the best of our knowledge, the asymptotic behavior under a sequence
of parameters of the Wald and likelihood ratio test statistics for a scale
parameter, has not been described for settings in which the Fisher information
has any rank less than full.

Our work suggests several avenues for future research: First, more theory is
needed on settings with nuisance parameters not orthogonal to the parameters of
interest. Existing theory suggests replacing the Fisher information by the
efficient Fisher information as we did in Section \ref{sec:datex}, but this has
not been formalized for inference near critical points. Based on simulations
(Supplementary Material) and intuition, we conjecture
our results may be adapted to settings where the block of the Fisher information
corresponding to the nuisance parameters is not nearly singular, but that there
may be additional challenges otherwise. Second, in some mixed models, for
example with crossed random effects, there are few independent observations even
as the total number of observations grows. Then, a different asymptotic theory
may be of interest. There are results on the consistency of maximum likelihood
estimators in such settings \citep{Jiang2013, Ekvall.Jones2020}, but the
properties of test statistics are largely unknown. Third, efficient software
implementations of the proposed method for popular mixed models are needed:
creating the proposed confidence region in practice often
requires inverting the test-statistic numerically which can be computationally
expensive. This is in contrast to the Wald statistic which can be inverted
analytically, but similarly to the likelihood ratio statistic which in general
requires numerical inversion. A natural starting point when implementing a
numerical procedure would be to consider a grid of parameter values centered at
some reasonable estimates. For example, even in non-linear mixed models where
maximum likelihood estimation can be computationally expensive, fast approximate
maximum likelihood estimates and Wald confidence regions are often available
through penalized quasi-likelihood or Laplace approximation of the likelihood.
Some further remarks on computing are in the Supplementary Material.

\bibliographystyle{apalike}
\bibliography{conf_crit}

\end{document}